\documentclass[a4paper,times,3p]{elsarticle}
\pdfoutput=1
\usepackage[utf8]{inputenc}

\usepackage[T1]{fontenc}

\usepackage{microtype}

\usepackage{booktabs}

\usepackage{amsmath}
\usepackage{amssymb}
\usepackage{amsthm}
\usepackage{bm}
\usepackage[hidelinks]{hyperref}
\usepackage{graphicx}

\usepackage{mathtools}

\usepackage[font=footnotesize,labelfont=bf]{caption}
\usepackage[font=footnotesize,labelfont=bf]{subcaption}
\usepackage{multirow}

\usepackage{pxfonts}

\usepackage{import}

\usepackage{gnuplot-lua-tikz}
\usepackage{tikz}
\usepackage{tikz-3dplot}
\usetikzlibrary{3d,calc}

\usepackage{aicescover}

\usepackage[textsize=scriptsize]{todonotes}
\usepackage{setspace}

\newtheorem{thm}{Theorem}
\newtheorem{lem}{Lemma}

\theoremstyle{remark}

\renewcommand{\Re}{\operatorname{Re}}
\renewcommand{\Im}{\operatorname{Im}}

\def\statement{\begin{minipage}[t]{.75\textwidth}
       NOTICE: This is the author's version of a work that was published in Journal of
       Non--Newtonian Fluid Mechanics 223 (2015) 209--220,  \href{http://dx.doi.org/10.1016/j.jnnfm.2015.07.004}{DOI:10.1016/j.jnnfm.2015.07.004}.
       \end{minipage}}

\makeatletter
\def\ps@pprintTitle{%
     \let\@oddhead\@empty
     \let\@evenhead\@empty
     \def\@oddfoot{\footnotesize\itshape
       \statement\hfill\today}%
     \let\@evenfoot\@oddfoot}
\makeatother

\journal{}

\aicescovertitle{The fully-implicit log-conformation formulation and its application to three-dimensional flows}
\aicescoverauthor{Philipp Knechtges}

\begin{document}
\aicescoverpage
\begin{frontmatter}

\title{The fully-implicit log-conformation formulation and its application to three-dimensional flows}
\author{Philipp Knechtges\corref{cor1}}
\ead{knechtges@cats.rwth-aachen.de}

\cortext[cor1]{Corresponding author}

\address{Chair for Computational Analysis of Technical Systems (CATS), RWTH Aachen University, 52056 Aachen, Germany\\
Center for Computational Engineering Science (CCES), RWTH Aachen University, 52056 Aachen, Germany}

\begin{abstract}
The stable and efficient numerical simulation of viscoelastic flows has been a constant struggle due to the High Weissenberg Number Problem.
While the stability for macroscopic descriptions could be greatly enhanced by the log-conformation method as proposed by Fattal and Kupferman, the application
of the efficient Newton--Raphson algorithm to the full monolithic system of governing equations, consisting of the log-conformation equations
and the Navier--Stokes equations, has always posed a problem. In particular, it is the formulation of the constitutive equations by means of the spectral
decomposition that hinders the application of further analytical tools.
Therefore, up to now, a fully monolithic approach could only be achieved in two dimensions, as, e.g., recently shown in
[P. Knechtges, M. Behr, S. Elgeti, Fully-implicit log-conformation formulation of constitutive laws,
J. Non-Newtonian Fluid Mech. 214 (2014) 78--87].

The aim of this paper is to find a generalization of the previously made considerations to three dimensions, such that a monolithic
Newton--Raphson solver based on the log-conformation formulation can be implemented also in this case.
The underlying idea is analogous to the two-dimensional case, to replace the eigenvalue decomposition in
the constitutive equation by an analytically more ``well-behaved'' term and to rely on the eigenvalue decomposition only for the actual computation.
Furthermore, in order to demonstrate the practicality of the proposed method, numerical results of the newly derived formulation are presented
in the case of the sedimenting sphere and ellipsoid benchmarks for the Oldroyd-B and Giesekus models.
It is found that the expected quadratic convergence of Newton's method can be achieved.
\end{abstract}

\end{frontmatter}

\section{Introduction}
Viscoelastic flows are ubiquitous in modern industrial applications. They are essential for the correct description
of the flow properties of blood, as well as polymer melts, which makes a good understanding of the used
models necessary for applications ranging from the design of blood pumps \cite{Marsden2014} to the layout of extrusion dies in plastics
manufacturing \cite{Pauli2013}.

Considering the demands stemming from the non-linear behavior of most of the used models and,
at the same time, the possibilities given through the advent of the computer age, it has become more and more
common not to base the model analysis solely on pure analytic grounds, but also to perform numerical simulations, which can be
applied to almost arbitrary geometries and domains. In the past, the macroscopic descriptions have been quite dominant, whereas
micro-macro simulations based on stochastic differential equations are now gaining importance \cite{Owens2002,Griebel2014}.
Although the latter offer a greater flexibility with respect to the modeling of the underlying
molecular dynamics, the former are still quite popular due to their lower computational cost.
Since this is important for the application of the numerical methods to complex geometries, this paper seeks
a description in the macroscopic framework. More specifically, we will consider the Oldroyd-B \cite{Oldroyd1950}
and the Giesekus model \cite{Giesekus1982}. The applicability of our methods, however, is not limited to these two models.

Simultaneously to the advent of numerical methods in the simulation of viscoelastic models, the High Weissenberg Number
Problem arose \cite{Keunings1986,Owens2002}.
With the Weissenberg number being a dimensionless constant that weights the contribution of the viscoelastic equations
to the description of the full system, this abstract term expresses the empirical fact that, with increasing Weissenberg number,
numerical simulations tend to fail. In fact, the range of attainable Weissenberg numbers turned out to be quite often lower than what was
measured in experiments, thus reducing the predictive power of simulations.

The most recent and quite successful approaches tackling the High Weissenberg Number Problem are the log-conformation methods,
first considered in \cite{Fattal2004} in order to better resolve exponential stress-boundary layers in regions of high strain.
Although they do not solve the High Weissenberg Number Problem completely,
they address the subproblem that numerical simulations do not necessarily preserve the positive-definiteness
of the conformation tensor; a property always fulfilled by the undiscretized equations \cite{Hulsen1990}.
The latter was found to be crucial for a numerical simulation not to fail. The underlying idea of the log-conformation methods
is as simple as it is powerful:
The so far primal degree of freedom --- the conformation tensor $\bm{\sigma}$ --- is replaced by its logarithm $\bm{\Psi}$.
Hence, $\bm{\sigma}$ is obtained by means of the matrix-exponential function $\exp\bm{\Psi}$, which automatically ensures
that $\bm{\sigma}$ remains positive-definite.

This, however, comes at the cost of finding a suitable replacement for the corresponding constitutive equation.
The way the original method \cite{Fattal2004} pursues is rather unusual, compared to other partial differential
equations, in the requirement of
an eigenvalue decomposition of $\bm{\Psi}$.
In particular, it is this spectral decomposition that hinders the direct
application of the Newton--Raphson algorithm to the full set of partial differential equations.
More specifically, the Newton--Raphson method involves a determination
of derivatives with respect to the $\bm{\Psi}$ degrees of freedom, including the derivatives of the eigenvalues and eigenvectors.
Nevertheless, considering the derivatives of eigenvectors, it is known that they become singular in the case of degenerate eigenvalues due to
the ambiguity in the eigenvectors.
As a remedy for this and for the difficulty of taking the derivative of the matrix-exponential function, first attempts resorted to
the approximation of the Jacobian matrix by difference quotients \cite{Coronado2007,Damanik2010}.

Even though a first analysis was conducted for the two-dimensional Leonov model in \cite{Kwon2004},
it was not until the work in \cite{Knechtges2014} and \cite{Saramito2014} that the Jacobian matrix was derived by
pure analytic means in two dimensions for a broader class of models.
As a continuation of these earlier works, this paper is devoted to a generalization to arbitrary dimensions,
along which we will also bridge the gap between the two expositions in \cite{Knechtges2014} and \cite{Saramito2014}.

In particular, we will not only discuss the derivation of a new constitutive equation in the first section, but also describe
the numerical implications in the case of an implementation into an existing Galerkin/Least-Squares (GLS) Navier--Stokes solver
in the succeeding section. The results of this solver are subsequently used in
Section~\ref{sec:benchmark} to study the falling sphere benchmark, where a sphere of radius $R$ sediments along the centerline of
a tube of radius $2R$. In order to demonstrate the applicability to truly three-dimensional flows, a modification of the same benchmark
with a tri-axial ellipsoid is considered as well.

Although our motivation stems mostly from the numerical side, the proposed equations
are purely analytic and as such may also serve as a new tool in future analytic studies; to the author's knowledge, this is the first
time that the constitutive equations for the log-conformation formulation can be stated in a closed form in
this generality.

\section{Theory}\label{sec:theory}
The aim of this section is the derivation of an alternative constitutive equation with $\bm{\Psi}$ as a new primal variable.
Starting point is the original constitutive equation in terms of the conformation tensor $\bm{\sigma}$ and the velocity $\bm{u}$.
Both are fields that, given boundary and initial conditions, have to be determined over a time-span $[0,T]$
and a $d$-dimensional domain $\tilde{\Omega}\subset\mathbb{R}^d$.

Following the exposition in \cite{Knechtges2014}, we consider
constitutive models of the form
\begin{align}
	\partial_t \bm{\sigma} + (\bm{u}\cdot \nabla) \bm{\sigma} + [\bm{\sigma},\Omega(\bm{u})] - \varepsilon(\bm{u})\bm{\sigma}
			- \bm{\sigma}\varepsilon(\bm{u})
		=& - \frac{1}{\lambda} P(\bm{\sigma})\,,
	\label{eqn:gen_conf_form2}
\end{align}
where $\varepsilon(\bm{u}) = \frac{1}{2}\left(\nabla\bm{u} + \nabla\bm{u}^T\right)$ denotes the strain tensor,
$\Omega(\bm{u}) = \frac{1}{2}\left(\nabla\bm{u} - \nabla\bm{u}^T\right)$ the vorticity tensor, $\lambda$ the
relaxation time, and $P(\bm{\sigma})$ an analytic function.
The bracket $[\bm{\sigma},\Omega(\bm{u})]$ is the so-called commutator, which is defined as
\begin{align*}
	[\bm{\sigma},\Omega(\bm{u})] = \bm{\sigma}\Omega(\bm{u}) - \Omega(\bm{u})\bm{\sigma}\, .
\end{align*}
Common choices for $P(\bm{\sigma})$ are $P(\bm{\sigma}) = \bm{\sigma} -\bm{1}$, leading to the Oldroyd-B
model \cite{Oldroyd1950}, or ${P(\bm{\sigma}) = \bm{\sigma} -\bm{1} + \alpha (\bm{\sigma}-\bm{1})^2}$ with $\alpha \in[0,1]$
in the Giesekus model \cite{Giesekus1982}. Generalizations of the
subsuming methods to the Johnson-Segalman model, as, e.g., done in \cite{Saramito2014}, or other models are in
principle possible, but omitted here for the sake of brevity.

Since the velocity field $\bm{u}$ is not determined so far, we have to combine the constitutive equations with
the Navier--Stokes equations in order to obtain a complete system of partial differential equations.
More specifically, the Navier--Stokes equations are given by
\begin{equation}
\label{eqn:navierstokes}
\begin{gathered}
	\nabla\cdot \bm{u} = 0\\
	\rho (\partial_t + \bm{u}\cdot\nabla) \bm{u} + \nabla p - 2\, \mu_S \nabla\cdot\varepsilon(\bm{u})
		- \frac{\mu_P}{\lambda} \nabla \cdot (\bm{\sigma}-\bm{1}) = 0\, ,
\end{gathered}
\end{equation}
with density $\rho$, as well as solvent and polymeric viscosity constants $\mu_S$ and $\mu_P$, respectively.

Furthermore, we will introduce function spaces $\mathcal{H}$ and $\mathcal{H}'$, which for the moment
could be chosen as perfectly smooth, i.e., $\mathcal{H}=\mathcal{H}' = C^\infty([0,T]\times\overline{\tilde{\Omega}})$,
and the derived spaces
\begin{align*}
	H &= \mathcal{H}^{d\times d} \quad H_{sym} = \{\bm{X}\in H \mathrel{|} \bm{X}^T = \bm{X}\},\\
	H' &= \mathcal{H}'^{d\times d} \quad H_{sym}' = \{\bm{X}\in H' \mathrel{|} \bm{X}^T = \bm{X}\}\, .
\end{align*}

The central statement of this paper then reads:
\begin{thm}\label{thm:firstmainthm}
Let the velocity field $\bm{u}$ be given with $\varepsilon(\bm{u})\in H_{sym}'$ and $\Omega(\bm{u})\in H'$.
If $\bm{\Psi}\in H_{sym}$ satisfies
\begin{align}
	\label{eqn:logconf}
	\partial_t\bm{\Psi} + (\bm{u}\cdot \nabla) \bm{\Psi} + [\bm{\Psi},\Omega(\bm{u})]
			- \frac{1}{(2\pi i)^2} \int_\Gamma \int_\Gamma f(z-z') \frac{1}{z-\bm{\Psi}}
				\varepsilon(\bm{u}) \frac{1}{z'-\bm{\Psi}} dz\, dz'&=
			- \frac{1}{\lambda} P\left(e^{\bm{\Psi}}\right) e^{-\bm{\Psi}}
\end{align}
with
\begin{align*}
	f(x) =& x + \frac{2x}{e^{x}-1} = \frac{x}{\tanh (x/2)}\, ,
\end{align*}
and $\Gamma$ chosen as a closed path surrounding the spectrum of $\bm{\Psi}$ in $\{z\in\mathbb{C}\mathrel{|} |\Im(z)| < \pi\}$,
then $\bm{\sigma}=\exp\bm{\Psi}\in H_{sym}$ solves the original constitutive equation \eqref{eqn:gen_conf_form2}.
\end{thm}

\begin{figure}
\begin{center}
\small
\begin{tikzpicture}
	\tikzstyle{singularity}=[circle,fill=black,thick,inner sep=0pt,minimum size=1.5mm]

	\node (pihalf) [draw=black,fill=gray!50,circle,inner sep=0,minimum size=1.0cm] at (0,0) {};
	\node (desc) at (0.95,0.4) {$B_{\pi/2}(0)$};

	\draw[->] (-4.2,0) -- (4.2,0) node[right] {$\Re (z)$};
	\draw[->] (0,-2.7) -- (0,2.7) node[above] {$\Im(z)$};

	\node (twopi) [singularity, label=right:$2\pi$] at (0,2) {};
	\node (minustwopi) [singularity, label=right:$-2\pi$] at (0,-2) {};

	\draw[dashed, very thin] (-4.2,1) -- (4.2,1);
	\draw[dashed, very thin] (-4.2,-1) -- (4.2,-1);
	\draw[-,very thick] (-0.1,1) -- (0.1,1) node[above right] {$\pi$};
	\draw[-,very thick] (-0.1,-1) -- (0.1,-1) node[below right] {$-\pi$};

	\draw[->,thick] (30:3.75 and 0.75) arc (30:390:3.75 and 0.75) node[above right] {$\Gamma$};

	\draw[very thick] (-2.0,0) -- (2.0,0);
	\draw[very thick] (-1.95,0.1) -- (-2.0,0.1) -- (-2.0,-0.1) -- (-1.95,-0.1)
		node[below] {$-||\bm{\Psi}||_H$};
	\draw[very thick] (1.95,0.1) -- (2.0,0.1) -- (2.0,-0.1) -- (1.95,-0.1)
		node[below] {$\phantom{-}||\bm{\Psi}||_H$};
\end{tikzpicture}
\end{center}
\caption{Illustration of a particular choice of the integration path $\Gamma$, as used in Theorem~\ref{thm:firstmainthm}.
	Here, choosing $\Gamma$ as an ellipse, with semi-major axis greater than $||\bm{\Psi}||_H$ and semi-minor axis smaller than $\pi$,
	ensures that the spectrum of $\bm{\Psi}$ is enclosed by $\Gamma$, while the poles of $f$, especially $\pm 2\pi i$, do not
	contribute to the integral.}
\label{fig:complexplane}
\end{figure}
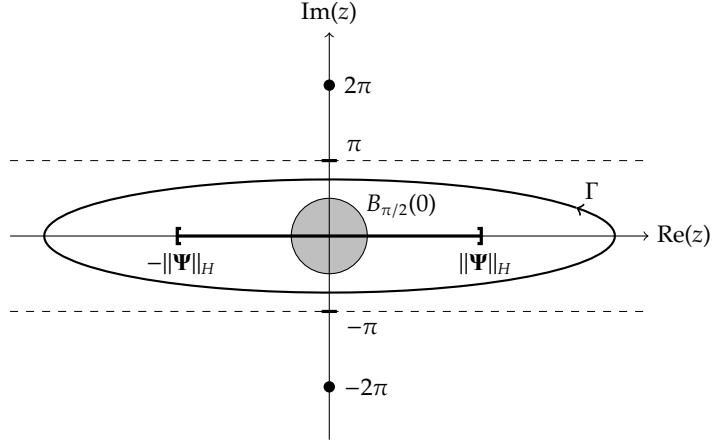

Before we come to the proof of this theorem, we need to consider certain properties of the relevant terms.
The first thing to notice in Eq.~\eqref{eqn:logconf} is the double integral, which is
similar to the familiar Cauchy integral from complex analysis.
One of the main differences to the ordinary Cauchy integral, however, is that the scalar ratio has been replaced by the resolvent
$1/(z-\bm{\Psi}) := (z\bm{1}-\bm{\Psi})^{-1}$, where $\bm{1}$ is the identity matrix and ${}^{-1}$~indicates the matrix inverse.
For smooth function spaces, it can be deduced
that, at a specific instant of space and time,
the resolvent exists if and only if $z$ does not equal any of the eigenvalues of $\bm{\Psi}(t,\bm{x})$, which are all real-valued.
Encircling these poles with our integration path $\Gamma$ subsequently gives us, by the same means as in the Cauchy integral setting,
some information on $f$ at these poles, but with the additional complexity that we have to deal with the matrix algebra.

Although we will further use this idea of numerically evaluating the integrals at the eigenvalues later on,
we will now leave the setting of smooth function spaces. Instead, we consider the
more general case of choosing $\mathcal{H}'$ as a Banach space and $\mathcal{H}\subset\mathcal{H}'$ as a commutative Banach algebra.
This opens up the door to a variety of spaces as they are used
in the analysis of partial differential equations. An example set of spaces in this more general setting would be the Sobolev-based spaces
\begin{align*}
	\mathcal{H} &= C^1([0,T],H^{s-1}(\tilde{\Omega}))\cap C^0([0,T],H^s(\tilde{\Omega})),\\
	\mathcal{H}'&= C^0([0,T],H^{s-1}(\tilde{\Omega}))\, ,
\end{align*}
with $s>d/2$ and $\tilde{\Omega}\subset \mathbb{R}^d$ being a Lipschitz-bounded domain \cite{Adams2003}.
It should be stressed that the mathematical discussion here is not limited to these spaces, and for the general requirements
on $\mathcal{H}$ and $\mathcal{H}'$ we refer to the appendix of \cite{Knechtges2014}.

Considering whether on these general spaces the double integral is well defined, the set of values
for which the resolvent is not defined is no longer restricted to the distinct eigenvalues, but may in fact be larger, although
still real-valued. Therefore, in order to
separate the terminology from the matrix algebra, this set is called spectrum in the general Banach algebra setting.
The generalization of Cauchy's integral to the theory of
Banach algebras\footnote{Throughout this paper we deliberately use the same symbol for the Banach algebra
as for its complexification.} is known as Dunford's integral and it is the essential ingredient
to define a functional calculus of holomorphic functions on these algebras \cite{Rudin,Yosida}.
More specifically, for a function $g$ that is holomorphic in the neighborhood of the spectrum of $\bm{\Psi}$, the $H$-valued
function is defined as
\begin{align}
	\label{eqn:dunfordintegral}
	g(\bm{\Psi}) :=& \frac{1}{2\pi i} \int_\Gamma \frac{g(z)}{z-\bm{\Psi}} dz\, ,
\end{align}
where $\Gamma$ is a contour surrounding the spectrum of $\bm{\Psi}$ within the same neighborhood.
Here, as well as in Theorem~\ref{thm:firstmainthm}, it is assumed that $\Gamma$ encircles the spectrum only once.
An immediate consequence of this definition is
that it allows us to explain the exponential function of $\bm{\Psi}$, which also has to be an element of $H_{sym}$.

A discrepancy between the integral in Theorem \ref{thm:firstmainthm} and the usual Dunford integral is that the argument of $f$
depends on the difference between the two integral variables $z$ and $z'$. The latter is also what makes it more difficult to ensure that
the poles of $f$, especially $\pm 2\pi i$, do not contribute to the integral.
In the formulation of Theorem~\ref{thm:firstmainthm} this has been realized by restricting the imaginary part of the integration
path $\Gamma$ to the region $|\Im{z}| < \pi$.
An example of a closed curve $\Gamma$ fulfilling the aforementioned criteria is depicted in Fig.~\ref{fig:complexplane}, where the fact is
also used that the spectrum of $\bm{\Psi}$ is always contained in the interval~$[-||\bm{\Psi}||_H,||\bm{\Psi}||_H]$.

Despite this minor restriction, most of the properties of Dunford's integral carry over to the double integral as well.
One of the more important features is the independence of the integral on the exact contour of $\Gamma$,
which is a key consequence of Cauchy's theorem \cite[Theorem 3.31]{Rudin}.

The final reason for not including the poles of $f$ in the integral in Theorem~\ref{thm:firstmainthm} is that we want to express $f$ by a Taylor series
in the course of the proof.
\begin{lem}\label{lem:series}
Let $g$ be a holomorphic function on a convex domain $\Omega'\subset\mathbb{C}$. Moreover, let the ball of radius $r$,  $B_r(0)$,
be contained in $\Omega'$. The Taylor series on this ball shall be given by $g(z) = \sum_{n=0}^\infty b_n z^n$.
Then for every $\bm{A}\in H$ with $||\bm{A}||_H < r/2$, $\bm{B}\in H'$, and $\Gamma \subset \frac{1}{2}\Omega'$ a contour around the
spectrum of $\bm{A}$, it holds
\begin{align*}
	 F(\bm{A},\bm{B}) :=& \frac{1}{(2\pi i)^2}
			\int_\Gamma \int_\Gamma g(z-z') \frac{1}{z-\bm{A}} \bm{B} \frac{1}{z'-\bm{A}} dz\, dz'\\
		 =&\sum_{n=0}^{\infty} b_n \{\bm{A},\bm{B}\}_{n}\, ,
\end{align*}
where $\{\bm{A},\bm{B}\}_{n}$ denotes the $n$-th iterated commutator
\begin{align*}
	\{\bm{A},\bm{B}\}_{n} :=& [\bm{A},\{\bm{A},\bm{B}\}_{n-1}] = \sum_{i=0}^n \binom{n}{i} (-1)^{i} \bm{A}^{n-i} \bm{B} \bm{A}^i\, .
\end{align*}
\end{lem}
\begin{proof}
Without loss of generality, Cauchy's theorem allows us to choose a contour $\Gamma$ within $B_{r/2}(0)$
that still surrounds the spectrum of $\bm{A}$.
Since the Taylor series converges uniformly on every compact subset of $B_r(0)$, and especially on $\Gamma-\Gamma\subset B_r(0)$,
one can deduce by similar means as for Dunford's integral (cf. \cite[Theorem 10.27]{Rudin}) that $g$ in $F(\bm{A},\bm{B})$
can be approximated by the Taylor series to yield an arbitrarily accurate approximation of $F(\bm{A},\bm{B})$.
As such, we can
assume $g(z-z') = (z-z')^n$. Furthermore, using a binomial expansion
$(z-z')^n = \sum_{i=0}^n \binom{n}{i} (-1)^i z^{n-i} {z'}^i$ we obtain
\begin{align*}
	& \frac{1}{(2\pi i)^2} \int_\Gamma \int_\Gamma g(z-z') \frac{1}{z-\bm{A}} \bm{B} \frac{1}{z'-\bm{A}} dz\, dz' \\
		=& \sum_{i=0}^n \binom{n}{i} (-1)^i \left(\frac{1}{2\pi i}\int_\Gamma \frac{z^{n-i}}{z-\bm{A}} dz\right)\, \bm{B}\,
			\left(\frac{1}{2\pi i} \int_\Gamma \frac{{z'}^{i}}{z'-\bm{A}} dz'\right)\, ,
\intertext{which together with \eqref{eqn:dunfordintegral}, or more rigorously \cite[Lemma 10.24]{Rudin}, yields the desired iterated commutator}
		=& \sum_{i=0}^n \binom{n}{i} (-1)^{i} \bm{A}^{n-i} \bm{B} \bm{A}^i
		= \{\bm{A},\bm{B}\}_{n}\, .
\end{align*}
\end{proof}

This lemma is already sufficient to prove Theorem~\ref{thm:firstmainthm} in the case $||\bm{\Psi}||_H<\pi$, as can be seen by
choosing $g(z) = f(z)=2\sum_{n=0}^\infty  \frac{B_{2n}}{(2n)!} z^{2n}$,
as well as $\Omega' = \mathbb{R}+i(-2\pi,2\pi)$ and $r=2\pi$.
Then it becomes apparent that if $\bm{\Psi}$ fulfills Eq.~\eqref{eqn:logconf}, it also has to
fulfill
\begin{align*}
	\partial_t\bm{\Psi} + (\bm{u}\cdot \nabla) \bm{\Psi} + [\bm{\Psi},\Omega(\bm{u})]
			- 2 \sum_{n=0}^\infty \frac{B_{2n}}{(2n)!} \{\bm{\Psi},\varepsilon(\bm{u})\}_{2n} &=
			- \frac{1}{\lambda} P\left(e^{\bm{\Psi}}\right) e^{-\bm{\Psi}}\, ,
\end{align*}
where $B_{2n}$ denote the even Bernoulli numbers.
This is exactly the equation for which the conclusion in Theorem \ref{thm:firstmainthm} has been proven in \cite[Theorem 1]{Knechtges2014}.

The generalization to $||\bm{\Psi}||_H \geq \pi$ is part of the
\begin{proof}[Proof of Theorem \ref{thm:firstmainthm}]
Assuming $\bm{\Psi}$ solves Eq.~\eqref{eqn:logconf},
in a first step the Wilcox Lemma \cite{Wilcox1966} is applied to handle the derivatives
$\partial_t + \bm{u}\cdot\nabla$ of the exponential mapping $\bm{\sigma} = \exp \bm{\Psi}$,
such that subsequently \eqref{eqn:logconf} can be inserted:
\begin{align*}
	\left(\partial_t + \bm{u}\cdot\nabla\right) \bm{\sigma} =& \int_0^1 e^{(1-\alpha)\bm{\Psi}}
			\left(\left(\partial_t + \bm{u}\cdot\nabla\right) \bm{\Psi}\right) e^{\alpha\bm{\Psi}}\, d\alpha\\
		=& - \frac{1}{\lambda} \int_0^1 e^{(1-\alpha)\bm{\Psi}}
				P\left(e^{\bm{\Psi}}\right) e^{-\bm{\Psi}} e^{\alpha\bm{\Psi}}\, d\alpha\\
			& - \int_0^1 e^{(1-\alpha)\bm{\Psi}} [\bm{\Psi},\Omega(\bm{u})] e^{\alpha\bm{\Psi}}\, d\alpha\\
			& + \int_0^1 e^{(1-\alpha)\bm{\Psi}} \frac{1}{(2\pi i)^2} \int_\Gamma \int_\Gamma f(z-z') \frac{1}{z-\bm{\Psi}}
				\varepsilon(\bm{u}) \frac{1}{z'-\bm{\Psi}} dz\, dz'\, e^{\alpha\bm{\Psi}}\, d\alpha\, .
\end{align*}
The integral involving $P\left(e^{\bm{\Psi}}\right)$ is the easiest to handle since all involved terms commute with each other, resulting in the contribution
$-\frac{1}{\lambda} P(\bm{\sigma})$. The vorticity term can be simplified by the fundamental theorem of calculus, which yields
\begin{align*}
	- \int_0^1 e^{(1-\alpha)\bm{\Psi}} [\bm{\Psi},\Omega(\bm{u})] e^{\alpha\bm{\Psi}}\, d\alpha
		=& \int_0^1 \partial_\alpha \left(e^{(1-\alpha)\bm{\Psi}} \Omega(\bm{u}) e^{\alpha\bm{\Psi}}\right)\, d\alpha
		= -e^{\bm{\Psi}} \Omega(\bm{u}) + \Omega(\bm{u}) e^{\bm{\Psi}} = - [\bm{\sigma}, \Omega(\bm{u})]\,.
\end{align*}
Hence, for $\bm{\sigma}$ to fulfill the original constitutive equation it is left to prove
\begin{align*}
	\int_0^1 e^{(1-\alpha)\bm{\Psi}} \frac{1}{(2\pi i)^2} \int_\Gamma \int_\Gamma f(z-z') \frac{1}{z-\bm{\Psi}}
				\varepsilon(\bm{u}) \frac{1}{z'-\bm{\Psi}} dz\, dz' e^{\alpha\bm{\Psi}}\, d\alpha
		=& \varepsilon(\bm{u}) e^{\bm{\Psi}} + e^{\bm{\Psi}} \varepsilon(\bm{u})\, .
\end{align*}
Rather than directly proving this equality we will follow the argumentation of Theorem 2 in \cite{Knechtges2014},
and consider a slightly more general equation where $\bm{\Psi}$ is replaced by $\beta\bm{\Psi}$,
such  that an analytic continuation argument in $\beta$ can be used to bridge the gap to the case $||\bm{\Psi}||_H < \pi$.
In particular, without loss of generality we will assume
that $\Gamma$, in addition to the spectrum of $\bm{\Psi}$, also encloses $B_{\pi/2}(0)$, as, e.g., depicted in Fig.~\ref{fig:complexplane}.
Our assertion now reads that
\begin{align}
	\int_0^1 e^{(1-\alpha)\beta\bm{\Psi}} F(\beta\bm{\Psi},\varepsilon(\bm{u})) e^{\alpha\beta\bm{\Psi}}\, d\alpha
		=& \varepsilon(\bm{u}) e^{\beta\bm{\Psi}} + e^{\beta\bm{\Psi}} \varepsilon(\bm{u})
\label{eqn:analyticcontformula1}
\end{align}
shall hold with
\begin{align*}
	F(\beta\bm{\Psi},\varepsilon(\bm{u})) :=& \frac{1}{(2\pi i)^2}
		\int_\Gamma \int_\Gamma f(z-z') \frac{1}{z-\beta\bm{\Psi}} \varepsilon(\bm{u}) \frac{1}{z'-\beta\bm{\Psi}} dz\, dz'
\end{align*}
for every $\beta$ in a sufficiently small simply-connected neighborhood $D$ of $[0,1]\cup B_{\pi/(2||\bm{\Psi}||_H)}(0)\subset\mathbb{C}$.

It is clear that the right-hand side of Eq.~\eqref{eqn:analyticcontformula1} is holomorphic for all $\beta$, due to
\begin{align*}
	\partial_\beta e^{\beta\bm{\Psi}} =& \bm{\Psi} e^{\beta\bm{\Psi}}\, .
\end{align*}
Additionally, using
\begin{align}\label{eqn:dervresolvent}
	\partial_\beta \frac{1}{z-\beta\bm{\Psi}} =& \frac{1}{z-\beta\bm{\Psi}} \bm{\Psi}  \frac{1}{z-\beta\bm{\Psi}}\, ,
\end{align}
it is evident that the left-hand side is holomorphic for $\beta \in D$.

Restricting ourselves for a moment to $|\beta| < \pi/(2||\bm{\Psi}||_H)$, we deduce from Lemma~\ref{lem:series}
that
\begin{align*}
	\int_0^1 e^{(1-\alpha)\beta\bm{\Psi}} F(\beta\bm{\Psi},\varepsilon(\bm{u})) e^{\alpha\beta\bm{\Psi}}\, d\alpha
		&= \int_0^1 e^{(1-\alpha)\beta\bm{\Psi}} \left( 2 \sum_{n=0}^\infty \frac{B_{2n}}{(2n)!}
					\{\beta \bm{\Psi}, \varepsilon(\bm{u})\}_{2n}\right)
					e^{\alpha\beta\bm{\Psi}}\, d\alpha\, ,
\end{align*}
which is essentially the form for which the identity \eqref{eqn:analyticcontformula1} has already been proven
in the proof of Theorem 2 in \cite{Knechtges2014}.
Thus, we are left with applying the monodromy theorem that asserts the uniqueness of the analytic continuation on $D$.
Thereby, \eqref{eqn:analyticcontformula1}
has in particular to hold for $\beta = 1$, and $\bm{\sigma}$ solves the original constitutive equation \eqref{eqn:gen_conf_form2}.
\end{proof}

\section{Numerical implementation}\label{sec:numimpl}

Given the newly derived constitutive equation \eqref{eqn:logconf},
we are going to discuss the numerical discretization in conjunction
with the Finite Element Method (FEM). The first part of this section will be
centered around the formulation of the discretized weak form in terms of space-time elements.
The second part will then be concerned with the linearization of the discretized weak form by means of the Newton--Raphson method.
In particular, it will also deal with the evaluation of the double integral and its derivatives.
It should be noted that the two subsections are only loosely coupled and that the discussion of the latter
subsection is not limited to the discretization scheme we have chosen, but may be easily generalized to other schemes.

\subsection{Discretization}
As in the preceding paper \cite{Knechtges2014}, we will use a mixture of a Streamline Upwind/Petrov-Galerkin (SUPG)- and
Galerkin/Least-Squares (GLS)-stabilized finite element method in combination with space-time
meshes to discretize the full monolithic system of constitutive equation \eqref{eqn:logconf} and Navier--Stokes equations \eqref{eqn:navierstokes}.
The SUPG method, which has originally been proposed in \cite{Brooks1982}, will serve as the stabilization method
of the advection-dominated constitutive equation, whereas a modified adjoint GLS will be used to stabilize the momentum equation
\cite{Franca1992a,Franca1992b,Behr1993}. The choice of a space-time method is mainly motivated by future
applications to deforming-domain problems.

Assuming a slicing of our space-time
domain $Q$ into $N$ slices $Q_n$, each spanning the computational domain from $t_n$ to $t_{n+1}$,
and furthermore a triangulation of $Q_n$
by the elements collected in $\mathcal{T}_{h,n}$, we introduce the function space
\begin{align*}
	V_{h,n} =& \left\{v\in C^0(\overline{Q_n}) \mathrel{}\middle|\mathrel{} \forall Q^e_n \in \mathcal{T}_{h,n} , v\circ T_{Q^e_n}
			\in \mathbb{P}_2\otimes \mathbb{P}_1\right\}\, .
\end{align*}
Here, the Lagrange elements $\mathbb{P}_2$ and $\mathbb{P}_1$ are employed in space and time, respectively, with $T_{Q^e_n}$ denoting
the isoparametric geometrical mapping from the reference element onto $Q^e_n$. For all applications within this paper the complete
space-time domain simplifies to $Q=[0,T] \times \tilde{\Omega}$ and the corresponding slices to $Q_n = [t_n,t_{n+1}) \times \tilde{\Omega}$.

Furthermore, defining the spatial boundary of the space-time slab as $P_n = \bigcup_{t\in [t_n,t_{n+1}]} \{t\}\times \partial\tilde{\Omega}_t$,
where $\tilde{\Omega}_t$ designates the spatial extent of the computational domain at a given instant of time $t$,
we use the following trial and test spaces
\begin{align}
	\mathcal{S}_{h,n} =& \left\{ (\bm{u},p,\bm{\Psi}) \in (V_{h,n})^d \times V_{h,n} \times (V_{h,n})^{d\cdot(d+1)/2} \mathrel{}\middle|\mathrel{}
				\bm{u}|_{P_{n,\bm{u}}} = \bm{g}_{\bm{u}}, \bm{\Psi}|_{P_{n,\bm{\Psi}}} = \bm{g}_{\bm{\Psi}}\right\}\\
	\mathcal{V}_{h,n} =& \left\{ (\bm{v},q,\bm{\Phi}) \in (V_{h,n})^d \times V_{h,n} \times (V_{h,n})^{d\cdot(d+1)/2} \mathrel{}\middle|\mathrel{}
				\bm{v}|_{P_{n,\bm{u}}} = 0, \bm{\Phi}|_{P_{n,\bm{\Psi}}} = 0\right\}\, ,
\end{align}
with $P_{n,\bm{u}}$ and $P_{n,\bm{\Psi}}$ being the subsets of $P_n$ on which $\bm{g}_{\bm{u}}$ and $\bm{g}_{\bm{\Psi}}$
are prescribed as Dirichlet boundary conditions. The full trial space, spanning the whole space-time domain, is hence
chosen as
\begin{align*}
	\mathcal{S}_{h} =& \left\{ (\bm{u},p,\bm{\Psi}) \in L^2(Q,\mathbb{R}^{d+1+d\cdot(d+1)/2}) \mathrel{}\middle|\mathrel{}
			(\bm{u},p,\bm{\Psi})|_{[t_n,t_{n+1})} \in \mathcal{S}_{h,n}\right\}\, .
\end{align*}

Using these definitions, the discretized weak problem can be formulated as follows: \textit{Given
the initial conditions $(\bm{u}^h)_0^- = \bm{u}_0$ and
$(\bm{\Psi}^h)_0^- = \bm{\Psi}_0$, we seek $\bm{z}^h = (\bm{u}^h,p^h,\bm{\Psi}^h) \in \mathcal{S}_{h}$ such that
on each time slab $Q_n$ with $0\leq n \leq N-1$ and for every $\bm{w}^h = (\bm{v}^h,q^h,\bm{\Phi}^h)\in \mathcal{V}_{h,n}$
the following equation is fulfilled:}
\begin{align}
\label{eqn:weakformdisc}
\begin{split}
	0 = a_n(\bm{w}^h, \bm{z}^h)
		:=& \quad \int_{Q_n} \bm{v}^h \cdot \rho \left(\partial_t \bm{u}^h + (\bm{u}^h \cdot \nabla) \bm{u}^h\right)
					+ \int_{Q_n} \frac{\mu_P}{\lambda} \varepsilon(\bm{v}^h) : \left(e^{\bm{\Psi}^h} -\bm{1}\right)\\&
				+ \int_{Q_n} 2\mu_s \varepsilon(\bm{v}^h) : \varepsilon(\bm{u}^h) - \int_{Q_n} (\nabla \cdot \bm{v}^h) \, p^h
					+ \int_{\tilde{\Omega}_n} (\bm{v}^h)_n^+ \cdot \rho \left((\bm{u}^h)_n^+ - (\bm{u}^h)_n^-\right) \\&
				+ \sum_e \int_{Q_n^e} \tau_{mom} \frac{1}{\rho} \left(\rho (\bm{u}^h \cdot \nabla) \bm{v}^h + \nabla q^h + \mu_S \Delta \bm{v}^h
						- \frac{\mu_P}{\lambda} \nabla \cdot \bm{\Phi}^h\right) \\&\qquad
							\cdot \left(\rho (\partial_t \bm{u}^h + (\bm{u}^h \cdot \nabla) \bm{u}^h)
						+ \nabla p^h - \mu_S \Delta \bm{u}^h - \frac{\mu_P}{\lambda} \nabla \cdot \left(e^{\bm{\Psi}^h}-\bm{1}\right)\right)\\&
				+ \int_{Q_n} q^h \, (\nabla\cdot \bm{u}^h)
				+ \int_{\tilde{\Omega}_n} (\bm{\Phi}^h)_n^+ : \frac{\mu_P}{2\lambda} \left((\bm{\Psi}^h)_n^+ - (\bm{\Psi}^h)_n^-\right)\\&
				+ \int_{Q_n} \frac{\mu_P}{2\lambda} \left(\bm{\Phi}^h + \tau_{cons} (\bm{u}^h \cdot\nabla) \bm{\Phi}^h\right) \\&\qquad
							: \left(\partial_t \bm{\Psi}^h + (\bm{u}^h\cdot \nabla) \bm{\Psi}^h + [\bm{\Psi}^h,\Omega(\bm{u}^h)]
								- F(\bm{\Psi}^h,\varepsilon(\bm{u}^h)) + \frac{1}{\lambda} P\left(e^{\bm{\Psi}^h}\right) e^{-\bm{\Psi}^h}\right)\, .
\end{split}
\end{align}
The inner product $\bm{\Phi}:\bm{\Psi}$ is as usual defined as $\mbox{tr}(\bm{\Phi}^T\bm{\Psi})$.
This weak form also incorporates concepts which are typical for space-time GLS realization, e.g.,
the weak coupling between the space-time slabs motivated by Discontinuous Galerkin methods.
Here, $(\bm{u}^h)_n^\pm$ is used as the short form for $\lim_{\xi\to 0} \bm{u}^h (t_n \pm \xi,\cdot)$ and $\tilde{\Omega}_n = \tilde{\Omega}_{t_n}$.

Furthermore, for all subsequent calculations within this paper the stabilization parameters were chosen as
\begin{align*}
	\tau_{mom} =& \mbox{min}\left(\rho \frac{h^2}{600\, \mu}, \frac{h}{2|\bm{u}|}, \frac{\Delta t}{2}\right)\, ,\\
	\tau_{cons} =& \mbox{min}\left(\left(2\frac{|\bm{u}|}{h} + \lambda^{-1}\right)^{-1}, \frac{\Delta t}{2}\right)\, ,
\end{align*}
where $\Delta t$ is the time-step size, $h$ the element diameter, $\mu = \mu_S+\mu_P$ the full viscosity, and $|\bm{u}|$
the absolute value of the velocity evaluated at the element center. In cases where stationary simulations were performed,
the corresponding parts of the discretized weak
form, namely, the explicit time-derivatives as well as the discontinuous coupling across space-time slabs, were neglected, which also applies to the
$\Delta t/2$ part of the stabilization constants. Similarly, in the creeping flow limit ($Re = 0$) the advective derivative
of the velocity $ (\bm{u}^h \cdot \nabla) \bm{u}^h$ was omitted from the momentum equation in conjunction with dropping $h/(2|\bm{u}|)$
from $\tau_{mom}$.

\subsection{Linearization and evaluation}\label{sec:linearization}
In a last step, the discretized weak form \eqref{eqn:weakformdisc} has to be linearized in order to make it amenable to linear solvers.
As already mentioned in the introduction, the used linearization method in this work is the Newton--Raphson algorithm,
which promises quadratic convergence
at the additional cost of providing a variational directional derivative of the weak form. More specifically, denoting the
directional derivative by
\begin{align*}
	\left. D a_n (\bm{w}^h, \cdot )\right|_{\bm{z}^h_{n,i}}\, \delta \bm{z}^h_{n,i} =&
			\left.\frac{d}{d\xi}\right|_{\xi=0} a_n(\bm{w}^h, \bm{z}^h_{n,i} + \xi \cdot \delta \bm{z}^h_{n,i})\, ,
\end{align*}
we iteratively solve
\begin{align*}
	\left. D a_n (\bm{w}^h, \cdot )\right|_{\bm{z}^h_{n,i}}\, \delta \bm{z}^h_{n,i}
		=& \, - a_n(\bm{w}^h, \bm{z}^h_{n,i})\quad \forall \bm{w}^h \in \mathcal{V}_{h,n}
\end{align*}
for $\delta \bm{z}^h_{n,i}\in\mathcal{V}_{h,n}$. The updated solution $\bm{z}^h_{n,i+1}$ can then be computed as
$\bm{z}^h_{n,i+1} = \bm{z}^h_{n,i} + \delta \bm{z}^h_{n,i}$.
The iteration is terminated as usual when the Euclidean norm of the residual $||\bm{r}||_2 := ||a_n(\cdot, \bm{z}^h_{n,i})||_2$
becomes smaller than a given threshold.

When the Newton--Raphson algorithm is employed in the context of the newly derived constitutive equation \eqref{eqn:logconf},
the immediate numerical implementation may lead to difficulties:
Due to their invariance on the exact contour of $\Gamma$, the evaluation of Cauchy-type integrals is prone to cancellation.
This applies to the double integral as well as to the exponential mapping.
The difficulty can be alleviated in the numerical setting by evaluating the integral directly or indirectly (e.g., through a quadrature rule)
only at specific instants of space and time. This condenses our Banach algebra to the usual matrix algebra, which essentially implies that
the spectrum of $\bm{\Psi}^h(t,\bm{x})$ contains at most up to $d$ distinct discrete points, i.e., the eigenvalues of $\bm{\Psi}^h(t,\bm{x})$.
Using the same techniques as are applied to identify the usual spectral decomposition method of interpreting matrix functions
with the Cauchy-type definition of matrix functions \eqref{eqn:dunfordintegral}, we will be able to reformulate the integral in the
framework of eigenvalues and eigenvectors.

For this, we introduce a set of $d$ eigenvalues $\lambda_i$ and $d$ eigenvectors $\bm{\tilde{e}}_i$ of $\bm{\Psi}^h(t,\bm{x})$, which are associated to a projection
operator $\bm{P}_i=\bm{\tilde{e}}_i\bm{\tilde{e}}_i^T$ that projects onto the one-dimensional subspaces spanned by the corresponding eigenvector.
Using this notation, linear algebra states that
\begin{align}\label{eqn:resolventeig}
	\frac{1}{z-\bm{\Psi}^h} =& \sum_{i=1}^d \frac{1}{z-\lambda_i} \bm{P}_i
\end{align}
has to hold, where for the sake of brevity the function arguments $(t,\bm{x})$ have been dropped.
Applying this equation to the double integral simplifies it to
\begin{align*}
	F(\bm{\Psi}^h,\varepsilon(\bm{u}^h)) =&\frac{1}{(2\pi i)^2} \int_\Gamma \int_\Gamma f(z-z') \frac{1}{z-\bm{\Psi}^h}
				\varepsilon(\bm{u}^h) \frac{1}{z'-\bm{\Psi}^h} dz\, dz'\\
		=& \sum_{i,j=1}^d \bm{P}_i \varepsilon(\bm{u}^h) \bm{P}_j \frac{1}{(2\pi i)^2} \int_\Gamma \int_\Gamma f(z-z')
				\frac{1}{z-\lambda_i} \frac{1}{z'-\lambda_j} dz\, dz'\, ,
\end{align*}
which then together with Cauchy's integral formula (or the residue theorem) yields
\begin{align}\label{eqn:evalF}
		F(\bm{\Psi}^h,\varepsilon(\bm{u}^h)) =& \sum_{i,j=1}^d f(\lambda_i - \lambda_j) \bm{P}_i \varepsilon(\bm{u}^h) \bm{P}_j\, .
\end{align}
It should be noted that this is the form of the $\varepsilon(\bm{u})$-term in the constitutive equation as it has been considered in \cite{Saramito2014}.

Of course, there is no doubt that, with the numerically well-studied QR-algorithms in mind, this form is much more suitable for numerical
evaluation. Nonetheless, it falls short in many applications when it comes to study perturbations of $\bm{\Psi}^h$, as it is the case
for the variational derivative needed in the Newton-Raphson algorithm. Existing implementations, as, e.g., in \cite{Knechtges2014,Saramito2014},
were limited to the two-dimensional case, since for a $2\times 2$ matrix it is still feasible to derive an algebraic closed expression for eigenvalues and eigenvectors
in dependence of $\bm{\Psi}^h$. Another approach would be general perturbation theory \cite{Kato}, which directly applies to
eigenvalues $\lambda_i$ and their projection operators $\bm{P}_i$, but this theory is prone to singularities in the vicinity
of degenerate eigenvalues.

The solution we will pursue here is similar to the general perturbation method in means of using complex calculus. As such we perform
the perturbation first in the framework of the double integral and then switch to the
eigenvalue representation. E.g., considering the variational derivative of the double integral $F(\bm{\Psi},\varepsilon(\bm{u}))$
with respect to $\bm{\Psi}^h$ in the direction $\delta\bm{\Psi}^h$, one obtains by similar means as in Eq.~\eqref{eqn:dervresolvent}
\begin{align*}
	\left. \frac{\partial}{\partial \xi}\right|_{\xi=0} F(\bm{\Psi}^h+\xi\,\delta\bm{\Psi}^h,\varepsilon(\bm{u}^h))
		=& \quad \frac{1}{(2\pi i)^2} \int_\Gamma \int_\Gamma f(z-z') \frac{1}{z-\bm{\Psi}^h} \delta\bm{\Psi}^h  \frac{1}{z-\bm{\Psi}^h}
				\varepsilon(\bm{u}^h) \frac{1}{z'-\bm{\Psi}^h} dz\, dz'\\
			& + \frac{1}{(2\pi i)^2} \int_\Gamma \int_\Gamma f(z-z') \frac{1}{z-\bm{\Psi}^h}
				\varepsilon(\bm{u}^h) \frac{1}{z'-\bm{\Psi}^h} \delta\bm{\Psi}^h  \frac{1}{z'-\bm{\Psi}^h} dz\, dz'\, .
\end{align*}
Inserting Eq.~\eqref{eqn:resolventeig} and applying the residue theorem then yields
\begin{align}\label{eqn:evaldervF}
	\left. \frac{\partial}{\partial \xi}\right|_{\xi=0} F(\bm{\Psi}^h+\xi\,\delta\bm{\Psi}^h,\varepsilon(\bm{u}^h))
		=& \sum_{i,j,k=1}^d \frac{f(\lambda_i-\lambda_k)-f(\lambda_j-\lambda_k)}{\lambda_i-\lambda_j}
				\left(\bm{P}_i \delta\bm{\Psi}^h \bm{P}_j \varepsilon(\bm{u}^h)  \bm{P}_k
					+ \bm{P}_k \varepsilon(\bm{u}^h) \bm{P}_j  \delta\bm{\Psi}^h \bm{P}_i\right)\, ,
\end{align}
where in accordance with the residue theorem the difference quotient has to be replaced by $f'(\lambda_i-\lambda_k)$ if $\lambda_i$
and $\lambda_j$ coincide.
It is clear that this formula can be evaluated along the same lines as the evaluation of the double integral itself \eqref{eqn:evalF}.

Similar considerations also yield the different derivatives of the exponential mapping as involved in the discretized weak form \eqref{eqn:weakformdisc}
\begin{align*}
	\left. \frac{\partial}{\partial \xi}\right|_{\xi=0} \exp(\bm{\Psi}^h+\xi\,\delta\bm{\Psi}^h)
		=& \sum_{i,j=1}^d \left(\frac{e^{\lambda_i}}{\lambda_i-\lambda_j} + \frac{e^{\lambda_j}}{\lambda_j-\lambda_i}\right)
			\bm{P}_i \delta\bm{\Psi}^h \bm{P}_j\\
		=& \sum_{i,j=1}^d e^{\lambda_i/2} e^{\lambda_j/2} \frac{\sinh((\lambda_i - \lambda_j)/2)}{(\lambda_i - \lambda_j)/2}
			\bm{P}_i \delta\bm{\Psi}^h \bm{P}_j\\
	\nabla \cdot \exp(\bm{\Psi}^h) =& \sum_{i,j,k=1}^d e^{\lambda_j/2} e^{\lambda_k/2} \frac{\sinh((\lambda_j - \lambda_k)/2)}{(\lambda_j - \lambda_k)/2}
			\bm{P}_j \partial_i\bm{\Psi}^h \bm{P}_k \bm{e}_i\, .
\end{align*}
Here, the vectors $\bm{e}_i$ denote the Cartesian basis vectors.
Additionally, due to the GLS stabilization, the variational derivative has to be considered for $\nabla \cdot \exp(\bm{\Psi}^h)$. The analysis yields
\begin{align*}
	\left. \frac{\partial}{\partial \xi}\right|_{\xi=0} \nabla \cdot \exp(\bm{\Psi}^h+\xi\,\delta\bm{\Psi}^h)
		=& \sum_{i,j,k=1}^d e^{\lambda_j/2} e^{\lambda_k/2} \frac{\sinh((\lambda_j - \lambda_k)/2)}{(\lambda_j - \lambda_k)/2}
				\bm{P}_j \partial_i\delta\bm{\Psi}^h \bm{P}_k \bm{e}_i\\
			&+ \sum_{i,j,k=1}^d \left(\frac{e^{\lambda_i}}{(\lambda_i-\lambda_j)(\lambda_i-\lambda_k)}
					+ \frac{e^{\lambda_j}}{(\lambda_j-\lambda_i)(\lambda_j-\lambda_k)}
					+ \frac{e^{\lambda_k}}{(\lambda_k-\lambda_i)(\lambda_k-\lambda_j)}\right)\\&\qquad\qquad
				\cdot \sum_{l=1}^d \left[\bm{P}_i \delta\bm{\Psi}^h \bm{P}_j \partial_l\bm{\Psi}^h \bm{P}_k \bm{e}_l
					+ \bm{P}_k \partial_l\bm{\Psi}^h \bm{P}_j \delta\bm{\Psi}^h \bm{P}_i \bm{e}_l\right]\, .
\end{align*}
For the numerical implementation, we will have to further rewrite the factor in the second sum, as it is in this form not appropriate
for evaluation
in the proximity of degenerate eigenvalues. Introducing auxiliary variables $x=(\lambda_i-\lambda_j)/3$, $y=(\lambda_i - \lambda_k)/3$, and
$z=(\lambda_j-\lambda_k)/3$, it can be reformulated as
\begin{align}
\begin{split}\label{eqn:evalexpd}
	& \frac{e^{\lambda_i}}{(\lambda_i-\lambda_j)(\lambda_i-\lambda_k)}
			+ \frac{e^{\lambda_j}}{(\lambda_j-\lambda_i)(\lambda_j-\lambda_k)}
			+ \frac{e^{\lambda_k}}{(\lambda_k-\lambda_i)(\lambda_k-\lambda_j)}\\
		=& \frac{1}{9} e^{\lambda_i/3} e^{\lambda_j/3} e^{\lambda_k/3} \left[ \frac{e^x -1}{x} \frac{e^y-1}{y}
			+ \frac{e^{-x} -1}{-x} \frac{e^z-1}{z}
			+ \frac{e^{-y} -1}{-y} \frac{e^{-z}-1}{-z}\right.\\&\qquad\left.
			+ \frac{1}{y-x}\left(\frac{e^{-x} -1}{-x} - \frac{e^{-y} -1}{-y}\right)
			+ \frac{1}{x+z}\left(\frac{e^x -1}{x} - \frac{e^{-z} -1}{-z}\right)
			+ \frac{1}{y-z}\left(\frac{e^y -1}{y} - \frac{e^z -1}{z}\right)\right]\, .
\end{split}
\end{align}
Thus, the evaluation is once again reduced to difference quotients. The latter, which already appeared in Eq.~\eqref{eqn:evaldervF},
can be easily approximated by a Taylor series in the vicinity of vanishing denominators
\begin{align*}
	\frac{g(x)-g(y)}{x-y} =& g^{(1)}\left(\frac{x+y}{2}\right)
		+ \frac{(x-y)^2}{24} g^{(3)} \left(\frac{x+y}{2}\right) + \mathcal{O}\left((x-y)^4\right)\, .
\end{align*}
Here $g(x)$ can stand either for $f(x)$ as in the case of Eq.~\eqref{eqn:evaldervF} or for $(e^x-1)/x$ as in Eq.~\eqref{eqn:evalexpd}.
In the former case the Taylor series is used for $|x-y| < 10^{-2}$ and in the latter case for $|x-y| < 10^{-3}$
in order to account for the use of floating point arithmetic with double precision.
In addition, the derivatives of $g(x)$ also have to be approximated for small values of $x$, which for $f(x)$ is performed
beneath $|x| < 10^{-1}$ and for $(e^x-1)/x$ below $|x| < 10^{-3}$. The numbers were obtained by comparing results
of a naive double precision implementation with the results computed in a much higher precision \cite{mpmath},
and afterwards choosing the Taylor polynomials
such that the relative error should not exceed $\sim 10^{-10}$. A higher precision is also quite unlikely to be needed as both
terms enter the linear equation system only on the left-hand side,
such that they only influence the convergence, but not the accuracy of the solution.
As we will later see, the convergence is influenced even more by the inexact solution of the
equation system through iterative linear solvers.

It should be noted that such a special treatment for the other functions involved is in general unnecessary.
The hyperbolic functions $\tanh(x/2)$ and
$\sinh(x/2)$ can be readily used if $x/2 \neq 0$, assuming their implementation is correctly rounded close to $0$ and the rounding
mode is set to round to nearest \cite{Muller2010}.
For $(e^x-1)/x$, the accuracy near $x\approx 0$ can be greatly improved by using a trick \cite{Higham2002} and evaluating
\begin{align*}
	\frac{e^x-1}{x} &= \frac{y-1}{\log y}\quad \mbox{with } y = e^x\, ,
\end{align*}
which is substituted by $1$ in the case of $y=1$.

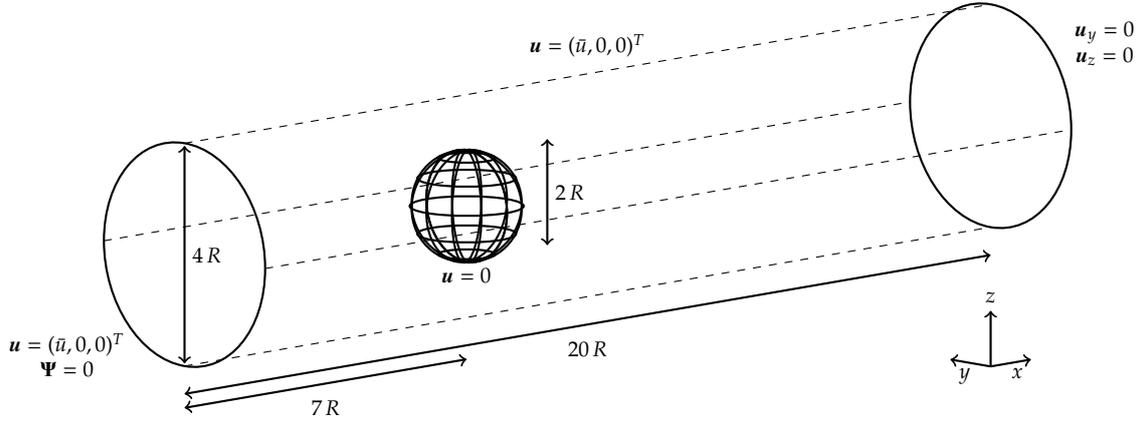
\begin{figure}[!t]
\footnotesize
\begin{center}
\tdplotsetmaincoords{80}{-45}
\begin{tikzpicture}[tdplot_main_coords,scale=0.75]
	\tikzstyle{silh}=[thin,dashed]
	\tikzstyle{sline}=[thick]
	\tikzstyle{axis}=[thick,->]

	\coordinate (O) at (7,0,0);
	\foreach \angle in {-90,-60,...,90}
	{
		\tdplotsinandcos{\sintheta}{\costheta}{\angle}%
		\coordinate (P) at (7,0,\sintheta);
		\tdplotdrawarc[sline]{(P)}{\costheta}{0}{360}{}{}
		\tdplotsetthetaplanecoords{\angle}
		\tdplotdrawarc[sline,tdplot_rotated_coords]{(O)}{1}{0}{360}{}{}
	}
	\node (sphere) at (7,0,-1.25) {$\bm{u} = 0$};

	\tdplotsetthetaplanecoords{90}
	\tdplotdrawarc[sline,tdplot_rotated_coords]{(0,0,0)}{2}{0}{360}{}{}
	\node[align=center] (inflow) at ($(-1,{sqrt(2)+0.5},{-sqrt(2)-0.5})$) {$\bm{u} = (\bar{u},0,0)^T$\\$\bm{\Psi}=0$};
	\tdplotsetthetaplanecoords{90}
	\tdplotdrawarc[sline,tdplot_rotated_coords]{(0,0,-20)}{2}{0}{360}{}{}
	\node[right=1.5,align=center] (outflow) at ($(20.5,{-sqrt(2)},{sqrt(2)})$) {$\bm{u}_y = 0$\\$\bm{u}_z = 0$};

	\draw[silh] (0,0,2) -- node[above=0.5] {$\bm{u} = (\bar{u},0,0)^T$} (20,0,2);
	\foreach \angle in {90,180,...,270}
	{
		\tdplotsinandcos{\sintheta}{\costheta}{\angle}%
		\draw[silh] (0,2*\sintheta,2*\costheta) -- (20,2*\sintheta,2*\costheta);
	}

	\draw[sline,<->] (0,0,-2.5) -- node[below=0.5] {$20\, R$} (20,0,-2.5);
	\draw[sline,<->] (0,0,-2.75) -- node[below=0.5] {$7\, R$} (7,0,-2.75);
	\draw[sline,<->,shorten >=1pt,shorten <=1pt] (0,0,-2) -- node[right=0.5] {$4\, R$} (0,0,2);
	\draw[sline,<->,shorten >=1pt,shorten <=1pt] (9,0,-1) -- node[right=0.5] {$2\, R$} (9,0,1);

	\draw[axis] (20,0,-4.5) -- (21,0,-4.5) node[anchor=north east]{$x$};
	\draw[axis] (20,0,-4.5) -- (20,1,-4.5) node[anchor=north west]{$y$};
	\draw[axis] (20,0,-4.5) -- (20,0,-3.5) node[anchor=south]{$z$};
\end{tikzpicture}
\end{center}
\caption{Illustration of the geometry and prescribed boundary conditions for the simulation of a uniform flow past a static sphere of radius $R$.}
\label{fig:spheregeom}
\end{figure}

All other terms on the left-hand side of the linear equation system arising from \eqref{eqn:weakformdisc} can be derived as usual.
The derivatives originating from the stabilization terms,
in particular the derivatives of $\tau_{cons} (\bm{u}^h \cdot\nabla) \bm{\Phi}^h$ with respect to $\bm{u}^h$,
are typically omitted, as they decrease the
robustness of the Newton--Raphson algorithm. Nonetheless, they are an essential ingredient for quadratic convergence, which will also be
discussed in the next section.

In our implementation we use an ILUT-preconditioned FGMRES implementation to solve the resulting linear systems
(cf. \cite{Saad1993,Saad1994}).

\section{Benchmarks}
\label{sec:benchmark}

In this section we will use the newly derived method to study two benchmarks:
the sedimenting sphere benchmark and a variation thereof with the sphere replaced by a tri-axial ellipsoid.

\subsection{Sedimenting sphere}

\begin{figure}[t]
\begin{center}
\includegraphics[width=\textwidth]{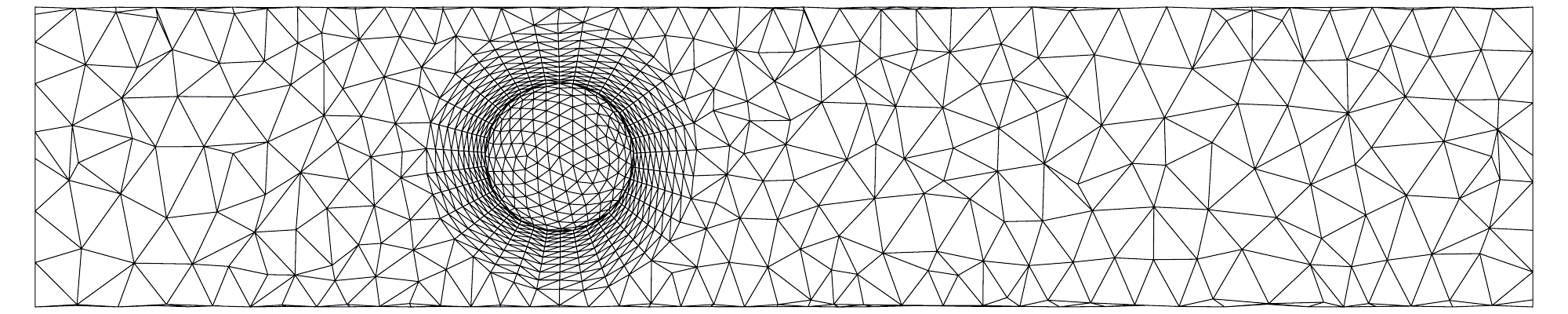}
\end{center}
\caption{Cut through the $xy$-plane of Mesh M1.}
\label{fig:cutmeshm1}
\end{figure}

\begin{table}[t]
\footnotesize
\centering
\begin{tabular}{lccc}
\toprule
 & M1 & M2 & M3 \\
\midrule
Number of elements on the sphere & 676 & 2602 & 9432 \\
Total number of nodes & 42788 & 217789 & 970454 \\
Total number of elements & 29791 & 157757 & 714417 \\
Krylov-space dimension & 150 & 300 $\mathrel{/}$ 350 & 400 \\
ILUT maximal fill-in $n_{ILUT}$ & 120 $\mathrel{/}$ 200 & 120 $\mathrel{/}$ 200 $\mathrel{/}$ 250 & 120 $\mathrel{/}$ 200 $\mathrel{/}$ 250 \\
ILUT threshold & $10^{-4}$ & $10^{-4}$ & $10^{-4}$ \\
Number of cores & 32 & 256  $\mathrel{/}$ 512 & 2048\\
\bottomrule
\end{tabular}
\caption{Mesh properties and solver parameters for the sedimenting sphere benchmark.}
\label{tab:meshproperties}
\end{table}

The sedimenting sphere in a tube benchmark is, in addition to the drag on confined cylinder benchmark, one of the classic benchmarks that has been
used in the past to measure the performance of numerical codes and different constitutive models.
It has been intensively studied experimentally, as well as numerically,
where numerous results for the upper-convected Maxwell model were obtained.
For a thorough review of the two aspects we refer to \cite{Mckinley2002,Owens2002} and the references
therein. Our analysis will be mostly centered around the Oldroyd-B model, which was already analyzed in
\cite{Lunsmann1993,Owens1996,Chauviere2000,Fan2003}, as well as the Giesekus model. In contrast to the just-mentioned literature, we will
not exploit the rotational symmetry in order to perform an in essence two-dimensional simulation of the three-dimensional problem, but will
solve the problem in three dimensions. The latter, although computationally more expensive, is of course more flexible
and preferred with the view on future applications.
As is commonly done, we will furthermore restrict ourselves to the simulation of the fully-developed flow condition,
where the sphere is sedimenting at constant speed, such that through a shift into the reference frame of the sphere,
we can reformulate the problem as
a stationary problem of a sphere at rest within a flow with uniform velocity $\bar{u}$.
Moreover, the gravitational force is neglected; with the exception of a missing buoyancy term in the pressure $p$,
this will not lead to any change of the flow field.

The geometry, as illustrated in Fig.~\ref{fig:spheregeom}, features a sphere of radius $R$. The sphere is located in the center of
a tube with radius $2R$ and is exposed to a uniform stream $\bm{u}=(\bar{u},0,0)^T$.
Based on the geometry, the flow conditions, and the relaxation time $\lambda$,
we define the Weissenberg number as
\begin{align*}
	Wi =& \frac{\lambda \bar{u}}{R}\, .
\end{align*}
It should be mentioned that the choice of the flow in $x$-direction is solely for the purpose of a better illustration.
The boundary conditions, as shown in Fig.~\ref{fig:spheregeom}, are a no-slip condition on the sphere, a
uniform stream of stress-free polymers ($\bm{\Psi}=0$) at the
inflow, and vanishing velocities perpendicular to the symmetry axis on the outflow. In accordance with the literature, only the
creeping flow limit ($Re=0$) is considered and the viscosity ratio is, in all conducted
simulations, chosen as $\beta = \mu_S/\mu = 0.5$.

As already indicated in the previous section, a tetrahedral $\mathbb{P}_2$ mesh was used to discretize the domain. A cut through the coarsest of the used
meshes can be seen in Fig.~\ref{fig:cutmeshm1}. All meshes feature a $0.9R$-thick boundary layer around the sphere in order to
properly resolve steep gradients. Further mesh properties as well as the linear solver parameters can be taken from Tab.~\ref{tab:meshproperties}.
Moreover, during the calculations the Weissenberg number was consecutively increased in such a way that the last result always served
as an initial guess for the following run.
Run times for a single simulation range approximately from $400\,\mbox{s}$ to $700\,\mbox{s}$ wall-clock time for the Meshes M1
and M2 on the Intel-based RWTH cluster.
For the finest mesh, the IBM-based Juqueen computer was used, resulting in run times of $1300 - 2200 \,\mbox{s}$.

\subsubsection{Oldroyd-B model}

\begin{table}[t]
\footnotesize
\centering
\begin{tabular}{cccccccc}
\toprule
\multirow{2}[3]{*}{$Wi$} & \multicolumn{6}{c}{$K$}\\
\cmidrule(lr){2-8}
 & M1 & M2 & M3 & \cite{Lunsmann1993} &  \cite{Owens1996} & \cite{Chauviere2000} & \cite{Fan2003} \\
\midrule
0.1 & 5.90022 & 5.90472 & 5.90576 & & & & \\
0.2 & 5.80240 & 5.80646 & 5.80763 & & & & \\
0.3 & 5.68858 & 5.69227 & 5.69356 & 5.69368 & 5.6963 & & \\
0.4 & 5.58068 & 5.58390 & 5.58527 & & & & \\
0.5 & 5.48692 & 5.48953 & 5.49093 & & & 5.4852 & \\
0.6 & 5.40899 & 5.41086 & 5.41227 & 5.41225 & 5.4117 & 5.4009 & \\
0.7 & 5.34592 & 5.34700 & 5.34838 & & & 5.3411 & \\
0.8 & 5.29582 & 5.29616 & 5.29747 & & & 5.2945 & \\
0.9 & 5.25660 & 5.25639 & 5.25761 & 5.25717 & & 5.2518 & \\
1.0 & 5.22628 & 5.22586 & 5.22700 & & & 5.2240 & \\
1.1 & 5.20312 & 5.20292 & 5.20402 & & & 5.2029 & \\
1.2 & 5.18568 & 5.18619 & 5.18733 & 5.18648 & & 5.1842 & 5.1877 \\
1.3 & 5.17278 & 5.17449 & 5.17581 & & & & 5.1763 \\
1.4 & 5.16354 & 5.16689 & 5.16851 & & & & \\
1.5 & 5.15723 & 5.16261 &               & 5.15293 & & & \\
\bottomrule
\end{tabular}
\caption{Results for the correction factor $K$ of the drag on the sphere when using the Oldroyd-B model.}
\label{tab:dragsphereoldroydb}
\end{table}

\begin{figure}[!t]
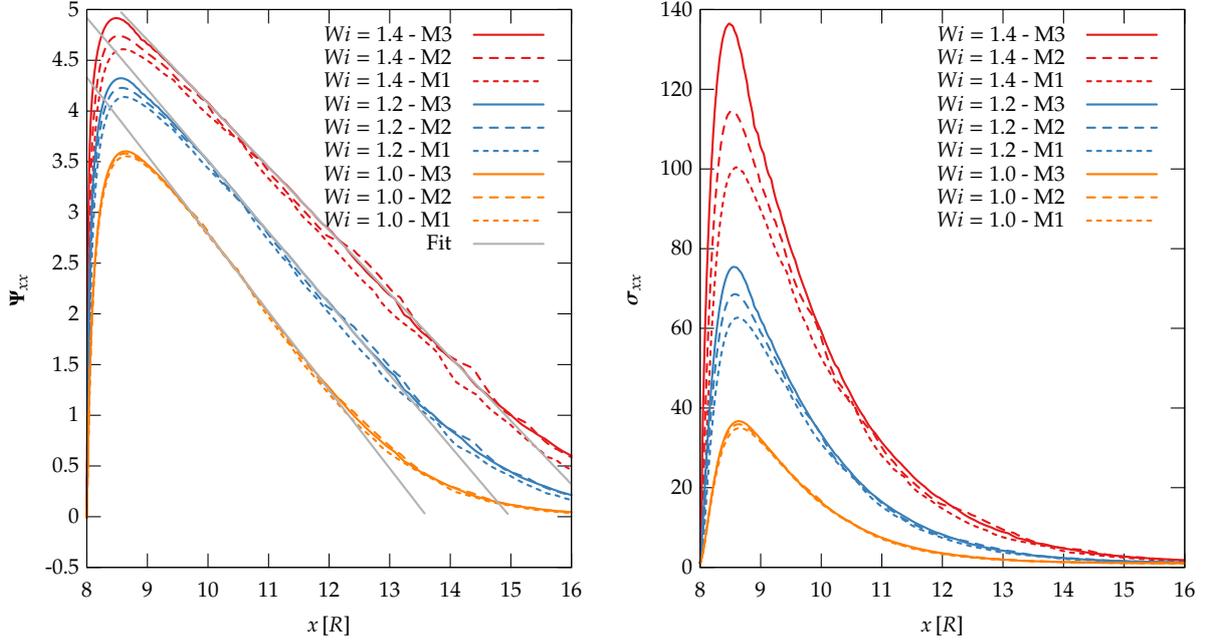

\footnotesize
\centering
\noindent
\begin{subfigure}[b]{0.49\textwidth}
\import{Fig/}{SpherePsixxWake.tex}
\end{subfigure}
\begin{subfigure}[b]{0.49\textwidth}
\import{Fig/}{SphereSigmaxxWake.tex}
\end{subfigure}
\caption{$\bm{\Psi}_{xx}$ and $\bm{\sigma}_{xx}$ plotted along the centerline in the wake of the sphere.}
\label{fig:spherepsixx_oldroydb}
\end{figure}

The use of the described benchmark in conjunction with the Oldroyd-B model has been covered extensively in literature.
One widely recognized performance quantity is the wall correction factor $K$, which is given as the
ratio of the drag force on the sphere to the Stokesian drag of a sphere in an unbounded Newtonian fluid
\begin{align}
\label{eqn:dragcorrection}
	K =& \frac{1}{6\pi\mu R\bar{u}} \int_{\Gamma_{Sphere}} \bm{e}_x^T \left[-p\bm{1} + 2 \mu_S \varepsilon (\bm{u})
			+ \frac{\mu_P}{\lambda}\left(e^{\bm{\Psi}}-\bm{1}\right)\right] \bm{n}\, .
\end{align}
Here, $\bm{n}$ denotes the unit normal field on the sphere, as usual.

The results of the simulations as presented in Tab.~\ref{tab:dragsphereoldroydb} match the results in literature quite well:
Generally convergence can --- independent of mesh size --- be claimed up to a Weissenberg number of $Wi=1.4$.
Above this threshold, the conditioning of the linearized system regresses. This can be mitigated only to a certain extent by
an increased number of GMRES iterations and an increased ILUT fill-in, but otherwise usually leads to a failure of the simulation.

It should be noted that the drag might not necessarily be the best benchmark quantity to measure the performance of
numerical discretizations, which may yield perfect drag results while not being able to properly predict other important flow characteristics.
One of these characteristics is the extensional flow in the wake of the sphere, where the polymers are stretched along the flow direction.
This is of special importance for the Oldroyd-B model, which, as it corresponds to the microscopic Hookean-dumbbell model,
has the property that the solution blows up in a purely extensional flow if the extensional rate exceeds a critical point ---
simply put, the dumbbells become infinitely long. Although there has not been a conclusive proof in literature yet,
it is believed that a similar mechanism is also responsible for the limitation
in the Weissenberg number for the feasible simulations in the falling sphere benchmark. To highlight this similarity,
notice that symmetry dictates
$\bm{\Psi}$ and $\nabla\bm{u}$ to be diagonal along the centerline, such that the constitutive equation
of $\bm{\Psi}_{xx}$ in Eq.~\eqref{eqn:logconf} reduces to
\begin{align}\label{eqn:centerlinepsixx}
	\bm{u}_x \partial_x \bm{\Psi}_{xx} - 2 \partial_x \bm{u}_x = - \frac{1}{\lambda} \left(1-e^{-\bm{\Psi}_{xx}}\right)\, .
\end{align}
Considering that at any extremal point $x^*$ of $\bm{\Psi}_{xx}$ the derivative has to vanish, rearranging this equation yields
\begin{align*}
	\bm{\Psi}_{xx} (x^*) =& - \log \left(1-2\lambda\, \partial_x \bm{u}_x(x^*)\right)\, .
\end{align*}
Thus, with $\partial_x \bm{u}_x(x^*)$ approaching $1/(2\lambda)$ the component $\bm{\Psi}_{xx}$ blows up.
Of course, nothing particular on the behavior of $\partial_x \bm{u}_x$ can be inferred within the framework of one-dimensional analysis due
to the incompressibility constraint.

The $\bm{\Psi}_{xx}$ actually predicted by the simulation can be seen in Fig.~\ref{fig:spherepsixx_oldroydb}. One of the points that becomes directly
apparent is that in these simulations mesh convergence can only be claimed up to Weissenberg numbers $Wi=1.0-1.2$. Above these values, it seems that
despite a boundary-layer-resolving mesh, the fluid characteristics in that region still cannot be accurately described.
This effect becomes even more pronounced looking
at $\bm{\sigma}_{xx}$, which modulo numerical noise is given by $\bm{\sigma}_{xx} = \exp \bm{\Psi}_{xx}$ and is also
depicted in Fig.~\ref{fig:spherepsixx_oldroydb}. There, of course, a slight deviation of an already large $\bm{\Psi}_{xx}$ is further amplified by
the exponential function. Furthermore, it should also be noted that the point where $\bm{\Psi}_{xx}$ attains its maximum seems to reach its maximal
$x$-value at $Wi=0.8$. For higher Weissenberg numbers the maximum is then shifted again in the direction of the sphere.

Considering these peculiarities and the general non-linearity of the governing equations, it is even more remarkable that on the downward slope of
$\bm{\Psi}_{xx}$ in Fig.~\ref{fig:spherepsixx_oldroydb} the field shows a linear behavior. Performing a least squares fit of linear curves to the
simulated data in the region $x=10R-12R$ on Mesh M3 yields slopes of $m=-0.6259\, R^{-1}$ for $Wi=1.4$, $m=-0.7037\, R^{-1}$ for $Wi=1.2$, and
$m=-0.7708\, R^{-1}$ for $Wi=1.0$. The resulting linear curves can be examined in Fig.~\ref{fig:spherepsixx_oldroydb}. It is yet unclear
which mechanism leads to this linear behavior.

\subsubsection{Giesekus model}

\begin{table}[!tb]
\footnotesize
\centering
\begin{tabular}{cccccccccc}
\toprule
 & \multicolumn{9}{c}{$K$}\\
\cmidrule(lr){2-10}
$Wi$ & \multicolumn{3}{c}{$\alpha=0.001$} & \multicolumn{3}{c}{$\alpha=0.01$} & \multicolumn{3}{c}{$\alpha=0.1$}\\
\cmidrule(lr){2-4} \cmidrule(lr){5-7} \cmidrule(lr){8-10}
 & M1 & M2 & M3 & M1 & M2 & M3 & M1 & M2 & M3 \\
\midrule
0.1  & 5.89918 & 5.90369 & 5.90473 & 5.88997 & 5.89464 & 5.89573 & 5.81454 & 5.82032 & 5.82166 \\
0.2  & 5.79863 & 5.80274 & 5.80393 & 5.76691 & 5.77147 & 5.77275 & 5.56351 & 5.57002 & 5.57160 \\
0.3  & 5.68098 & 5.68479 & 5.68610 & 5.62095 & 5.62552 & 5.62694 & 5.31523 & 5.32188 & 5.32349 \\
0.4  & 5.56845 & 5.57185 & 5.57324 & 5.47847 & 5.48300 & 5.48451 & 5.09969 & 5.10625 & 5.10785 \\
0.5  & 5.46934 & 5.47222 & 5.47366 & 5.34928 & 5.35374 & 5.35531 & 4.91688 & 4.92331 & 4.92489 \\
0.6  & 5.38538 & 5.38763 & 5.38910 & 5.23534 & 5.23965 & 5.24127 & 4.76150 & 4.76781 & 4.76938 \\
0.7  & 5.31561 & 5.31714 & 5.31861 & 5.13538 & 5.13952 & 5.14118 & 4.62842 & 4.63461 & 4.63616 \\
0.8  & 5.25811 & 5.25891 & 5.26037 & 5.04716 & 5.05109 & 5.05280 & 4.51345 & 4.51955 & 4.52109 \\
0.9  & 5.21075 & 5.21091 & 5.21235 & 4.96839 & 4.97211 & 4.97387 & 4.41335 & 4.41936 & 4.42090 \\
1.0  & 5.17150 & 5.17122 & 5.17264 & 4.89714 & 4.90066 & 4.90248 & 4.32556 & 4.33150 & 4.33303 \\
1.1  & 5.13857 & 5.13811 & 5.13955 & 4.83192 & 4.83528 & 4.83716 & 4.24805 & 4.25394 & 4.25545 \\
1.2  & 5.11050 & 5.11012 & 5.11165 & 4.77169 & 4.77491 & 4.77684 & 4.17920 & 4.18503 & 4.18653 \\
1.3  & 5.08611 & 5.08608 & 5.08774 & 4.71569 & 4.71879 & 4.72077 & 4.11769 & 4.12347 & 4.12496 \\
1.4  & 5.06448 & 5.06501 & 5.06688 & 4.66336 & 4.66637 & 4.66840 & 4.06245 & 4.06818 & 4.06966 \\
1.5  & 5.04489 & 5.04617 & 5.04829 & 4.61432 & 4.61725 & 4.61931 & 4.01260 & 4.01828 & 4.01975 \\
1.6  &               & 5.02897 & 5.03139 & 4.56825 & 4.57111 & 4.57319 & 3.96740 & 3.97303 & 3.97448 \\
1.8  &               &               &               & 4.48402 & 4.48676 & 4.48886 & 3.88863 & 3.89413 & 3.89557 \\
2.0  &               &               &               & 4.40900 & 4.41166 & 4.41375 & 3.82234 & 3.82771 & 3.82914 \\
2.2  &               &               &               & 4.34189 & 4.34447 & 4.34653 & 3.76580 & 3.77103 & 3.77245 \\
2.4  &               &               &               & 4.28156 & 4.28409 & 4.28612 & 3.71703 & 3.72211 & 3.72352 \\
2.6  &               &               &               & 4.22709 & 4.22960 & 4.23159 & 3.67452 & 3.67943 & 3.68085 \\
3.0  &               &               &               & 4.13269 & 4.13526 & 4.13716 & 3.60400 & 3.60857 & 3.61001 \\
3.5  &               &               &               & 4.03595 & 4.03881 & 4.04057 & 3.53552 & 3.53969 & 3.54117 \\
4.0  &               &               &               & 3.95673 & 3.96008 & 3.96170 & 3.48206 & 3.48589 & 3.48740 \\
4.5  &               &               &               & 3.89055 & 3.89454 & 3.89602 & 3.43910 & 3.44264 & 3.44418 \\
5.0  &               &               &               & 3.83433 & 3.83905 & 3.84040 & 3.40376 & 3.40707 & 3.40864 \\
5.5  &               &               &               & 3.78591 & 3.79140 & 3.79265 & 3.37414 & 3.37728 & 3.37886 \\
6.5  &               &               &               & 3.70655 & 3.71357 & 3.71469 & 3.32716 & 3.33010 & 3.33166 \\
7.5  &               &               &               & 3.64403 & 3.65246 & 3.65355 & 3.29144 & 3.29432 & 3.29584 \\
8.0  &               &               &               & 3.61743 & 3.62651 & 3.62762 & 3.27658 & 3.27947 & 3.28096 \\
8.5  &               &               &               & 3.59334 & 3.60301 & 3.60415 & 3.26328 & 3.26619 & 3.26765 \\
9.0  &               &               &               & 3.57140 & 3.58160 & 3.58280 & 3.25130 & 3.25424 & 3.25567 \\
9.5  &               &               &               & 3.55134 & 3.56202 & 3.56328 & 3.24044 & 3.24344 & 3.24482 \\
10.0 &               &               &               & 3.53290 & 3.54402 & 3.54535 & 3.23055 & 3.23360 & 3.23495 \\
11.0 &               &               &               & 3.50015 &  & 3.51352 & 3.21318 & 3.21637 & 3.21764 \\
12.0 &               &               &               & 3.47192 &  & 3.48607 & 3.19839 & 3.20173 & 3.20293 \\
13.0 &               &               &               & 3.44729 &  & 3.46211 & 3.18564 & 3.18914 & 3.19027 \\
14.0 &               &               &               & 3.42558 &  & 3.44098 & 3.17452 & 3.17816 & 3.17923 \\
15.0 &               &               &               & 3.40628 &  & 3.42218 & 3.16472 & 3.16851 & 3.16952 \\
\bottomrule
\end{tabular}
\caption{Results for the correction factor $K$ of the drag on the sphere when using the Giesekus model.}
\label{tab:dragspheregiesekus}
\end{table}

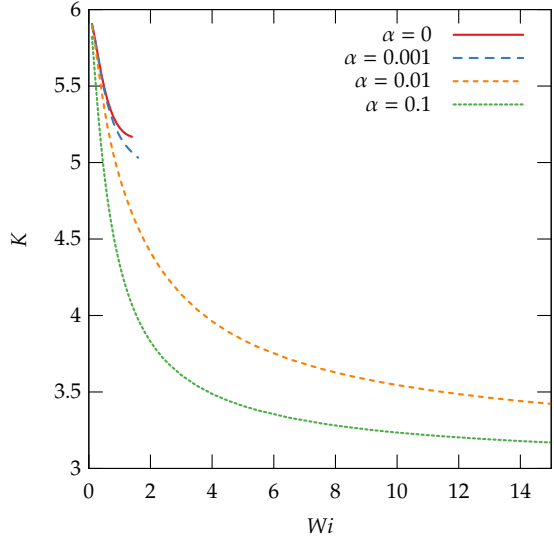
\begin{figure}[t]
\footnotesize
\centering
\noindent
\begin{tikzpicture}[gnuplot]
\gpcolor{color=gp lt color border}
\gpsetlinetype{gp lt border}
\gpsetlinewidth{1.00}
\draw[gp path] (2.592,0.985)--(2.772,0.985);
\draw[gp path] (8.676,0.985)--(8.496,0.985);
\node[gp node right] at (2.592,0.985) { 3};
\draw[gp path] (2.592,1.999)--(2.772,1.999);
\draw[gp path] (8.676,1.999)--(8.496,1.999);
\node[gp node right] at (2.592,1.999) { 3.5};
\draw[gp path] (2.592,3.013)--(2.772,3.013);
\draw[gp path] (8.676,3.013)--(8.496,3.013);
\node[gp node right] at (2.592,3.013) { 4};
\draw[gp path] (2.592,4.027)--(2.772,4.027);
\draw[gp path] (8.676,4.027)--(8.496,4.027);
\node[gp node right] at (2.592,4.027) { 4.5};
\draw[gp path] (2.592,5.041)--(2.772,5.041);
\draw[gp path] (8.676,5.041)--(8.496,5.041);
\node[gp node right] at (2.592,5.041) { 5};
\draw[gp path] (2.592,6.055)--(2.772,6.055);
\draw[gp path] (8.676,6.055)--(8.496,6.055);
\node[gp node right] at (2.592,6.055) { 5.5};
\draw[gp path] (2.592,7.069)--(2.772,7.069);
\draw[gp path] (8.676,7.069)--(8.496,7.069);
\node[gp node right] at (2.592,7.069) { 6};
\draw[gp path] (2.592,0.985)--(2.592,1.165);
\draw[gp path] (2.592,7.069)--(2.592,6.889);
\node[gp node center] at (2.592,0.677) { 0};
\draw[gp path] (3.403,0.985)--(3.403,1.165);
\draw[gp path] (3.403,7.069)--(3.403,6.889);
\node[gp node center] at (3.403,0.677) { 2};
\draw[gp path] (4.214,0.985)--(4.214,1.165);
\draw[gp path] (4.214,7.069)--(4.214,6.889);
\node[gp node center] at (4.214,0.677) { 4};
\draw[gp path] (5.026,0.985)--(5.026,1.165);
\draw[gp path] (5.026,7.069)--(5.026,6.889);
\node[gp node center] at (5.026,0.677) { 6};
\draw[gp path] (5.837,0.985)--(5.837,1.165);
\draw[gp path] (5.837,7.069)--(5.837,6.889);
\node[gp node center] at (5.837,0.677) { 8};
\draw[gp path] (6.648,0.985)--(6.648,1.165);
\draw[gp path] (6.648,7.069)--(6.648,6.889);
\node[gp node center] at (6.648,0.677) { 10};
\draw[gp path] (7.459,0.985)--(7.459,1.165);
\draw[gp path] (7.459,7.069)--(7.459,6.889);
\node[gp node center] at (7.459,0.677) { 12};
\draw[gp path] (8.270,0.985)--(8.270,1.165);
\draw[gp path] (8.270,7.069)--(8.270,6.889);
\node[gp node center] at (8.270,0.677) { 14};
\draw[gp path] (2.592,7.069)--(2.592,0.985)--(8.676,0.985)--(8.676,7.069)--cycle;
\node[gp node center,rotate=-270] at (1.702,4.027) {$K$};
\node[gp node center] at (5.634,0.215) {$Wi$};
\node[gp node right] at (7.208,6.735) {$\alpha=0$};
\gpcolor{rgb color={0.894,0.102,0.110}}
\gpsetlinetype{gp lt plot 0}
\gpsetlinewidth{2.00}
\draw[gp path] (7.392,6.735)--(8.308,6.735);
\draw[gp path] (2.633,6.878)--(2.673,6.679)--(2.714,6.448)--(2.754,6.228)--(2.795,6.037)%
  --(2.835,5.877)--(2.876,5.748)--(2.916,5.644)--(2.957,5.563)--(2.998,5.501)--(3.038,5.455)%
  --(3.079,5.421)--(3.119,5.398)--(3.160,5.383);
\gpcolor{color=gp lt color border}
\node[gp node right] at (7.208,6.427) {$\alpha=0.001$};
\gpcolor{rgb color={0.216,0.494,0.722}}
\gpsetlinetype{gp lt plot 1}
\draw[gp path] (7.392,6.427)--(8.308,6.427);
\draw[gp path] (2.633,6.876)--(2.673,6.671)--(2.714,6.432)--(2.754,6.204)--(2.795,6.002)%
  --(2.835,5.830)--(2.876,5.687)--(2.916,5.569)--(2.957,5.472)--(2.998,5.391)--(3.038,5.324)%
  --(3.079,5.267)--(3.119,5.219)--(3.160,5.177)--(3.200,5.139)--(3.241,5.105);
\gpcolor{color=gp lt color border}
\node[gp node right] at (7.208,6.119) {$\alpha=0.01$};
\gpcolor{rgb color={1.000,0.498,0.000}}
\gpsetlinetype{gp lt plot 2}
\draw[gp path] (7.392,6.119)--(8.308,6.119);
\draw[gp path] (2.633,6.858)--(2.673,6.608)--(2.714,6.312)--(2.754,6.024)--(2.795,5.762)%
  --(2.835,5.530)--(2.876,5.327)--(2.916,5.148)--(2.957,4.988)--(2.998,4.843)--(3.038,4.711)%
  --(3.079,4.588)--(3.119,4.475)--(3.160,4.369)--(3.200,4.269)--(3.241,4.175)--(3.322,4.004)%
  --(3.403,3.852)--(3.484,3.716)--(3.565,3.593)--(3.647,3.483)--(3.809,3.291)--(4.012,3.095)%
  --(4.214,2.935)--(4.417,2.802)--(4.620,2.689)--(4.823,2.592)--(5.228,2.434)--(5.634,2.310)%
  --(5.837,2.258)--(6.040,2.210)--(6.242,2.167)--(6.445,2.127)--(6.648,2.091)--(7.054,2.026)%
  --(7.459,1.971)--(7.865,1.922)--(8.270,1.879)--(8.676,1.841);
\gpcolor{color=gp lt color border}
\node[gp node right] at (7.208,5.811) {$\alpha=0.1$};
\gpcolor{rgb color={0.302,0.686,0.290}}
\gpsetlinetype{gp lt plot 3}
\draw[gp path] (7.392,5.811)--(8.308,5.811);
\draw[gp path] (2.633,6.707)--(2.673,6.200)--(2.714,5.697)--(2.754,5.260)--(2.795,4.889)%
  --(2.835,4.573)--(2.876,4.303)--(2.916,4.070)--(2.957,3.867)--(2.998,3.688)--(3.038,3.531)%
  --(3.079,3.391)--(3.119,3.266)--(3.160,3.154)--(3.200,3.053)--(3.241,2.961)--(3.322,2.801)%
  --(3.403,2.666)--(3.484,2.552)--(3.565,2.452)--(3.647,2.366)--(3.809,2.222)--(4.012,2.082)%
  --(4.214,1.973)--(4.417,1.886)--(4.620,1.814)--(4.823,1.753)--(5.228,1.658)--(5.634,1.585)%
  --(5.837,1.555)--(6.040,1.528)--(6.242,1.503)--(6.445,1.481)--(6.648,1.461)--(7.054,1.426)%
  --(7.459,1.397)--(7.865,1.371)--(8.270,1.348)--(8.676,1.329);
\gpcolor{color=gp lt color border}
\gpsetlinetype{gp lt border}
\gpsetlinewidth{1.00}
\draw[gp path] (2.592,7.069)--(2.592,0.985)--(8.676,0.985)--(8.676,7.069)--cycle;
\gpdefrectangularnode{gp plot 1}{\pgfpoint{2.592cm}{0.985cm}}{\pgfpoint{8.676cm}{7.069cm}}
\end{tikzpicture}
\caption{Wall correction factor $K$ plotted for different values of the mobility $\alpha$, computed on the Mesh M3.}
\label{fig:dragsphere_gies}
\end{figure}

A similar analysis as for the Oldroyd-B model is also conducted for the Giesekus model that extends the Oldroyd-B model by
an additional term. In fact, the Oldroyd-B model is a special case of the Giesekus model for a vanishing mobility parameter $\alpha=0$.

As before, the quantity studied first is the drag correction factor $K$ for several Weissenberg numbers $Wi$
and varying mobility $\alpha$. The results are collected in Tab.~\ref{tab:dragspheregiesekus} and depicted in
Fig.~\ref{fig:dragsphere_gies}. It is notable that, for the smallest $\alpha=0.001$, the model exhibits similar numerical behavior as the
Oldroyd-B model, namely convergent results only up to $Wi=1.6$. This is to be expected,
higher Weissenberg numbers may be achievable with finer meshes in contrast to the Oldroyd-B model. Increasing $\alpha$ shows that
the drag on the sphere decreases in general, which is attributable to the shear-thinning properties of a Giesekus fluid.
Moreover, for all performed numerical calculations the drag is monotonically decreasing with
increasing Weissenberg number $Wi$, and there are indications that $K$ reaches a plateau for sufficiently high $Wi$.

\begin{figure}[!b]
\footnotesize
\centering
\noindent
\includegraphics{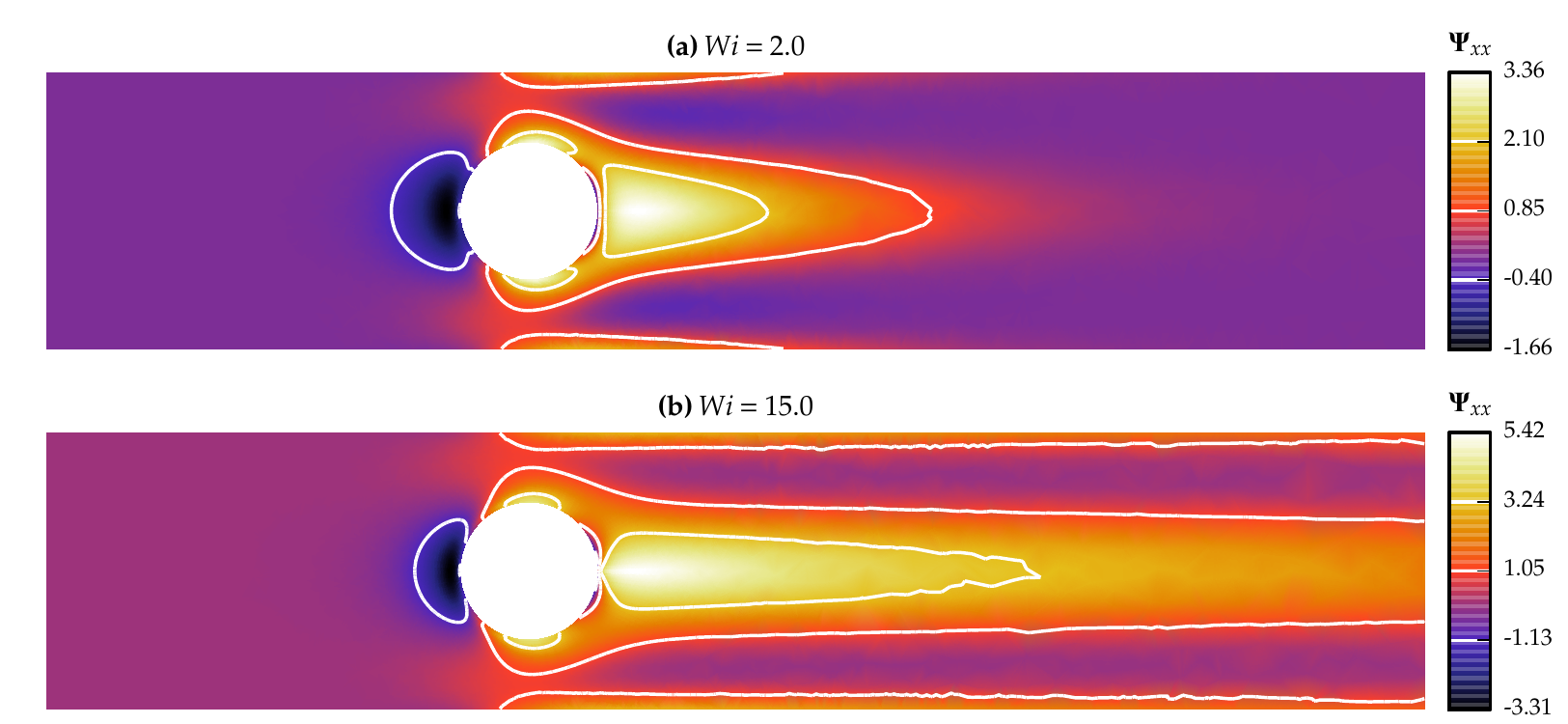}
\caption{Cut through the $xy$-plane of Mesh M3, illustrating $\bm{\Psi}_{xx}$ for different Weissenberg numbers and $\alpha = 0.1$.}
\label{fig:spherepsixx_a01}
\end{figure}

Looking at the extensional flow characteristics of the Giesekus model in the wake of the sphere, it already becomes apparent from the equivalent of
Eq.~\eqref{eqn:centerlinepsixx} that the model is better-behaved:
\begin{align*}
	\bm{u}_x \partial_x \bm{\Psi}_{xx} - 2 \partial_x \bm{u}_x = - \frac{1}{\lambda} \left(1-2\alpha
		-(1-\alpha) e^{-\bm{\Psi}_{xx}} + \alpha e^{\bm{\Psi}_{xx}}\right)\, .
\end{align*}
Here, the additional $\alpha \exp(\bm{\Psi}_{xx})$ term can potentially compensate an increase of $\partial_x \bm{u}_x$
exceeding $1/(2\lambda)$, thus limiting the increase of $\bm{\Psi}_{xx}$.
The resulting computations of $\bm{\Psi}_{xx}$ for two different Weissenberg numbers $Wi=2.0$ and $Wi=15.0$ are shown in Fig.~\ref{fig:spherepsixx_a01}.
The results reflect clearly that with increasing Weissenberg number, the polymers need more time to relax to their stress-free state,
which means that they are transported further downstream before they reach this state. As such, the $\bm{\Psi}_{xx}$ contours also extent further
downstream for higher Weissenberg numbers than for lower ones.
As a consequence the demands on the used geometry and meshes increase: They need
to sustain a high refinement level over a larger region in the wake of the sphere.

\begin{figure}[t]
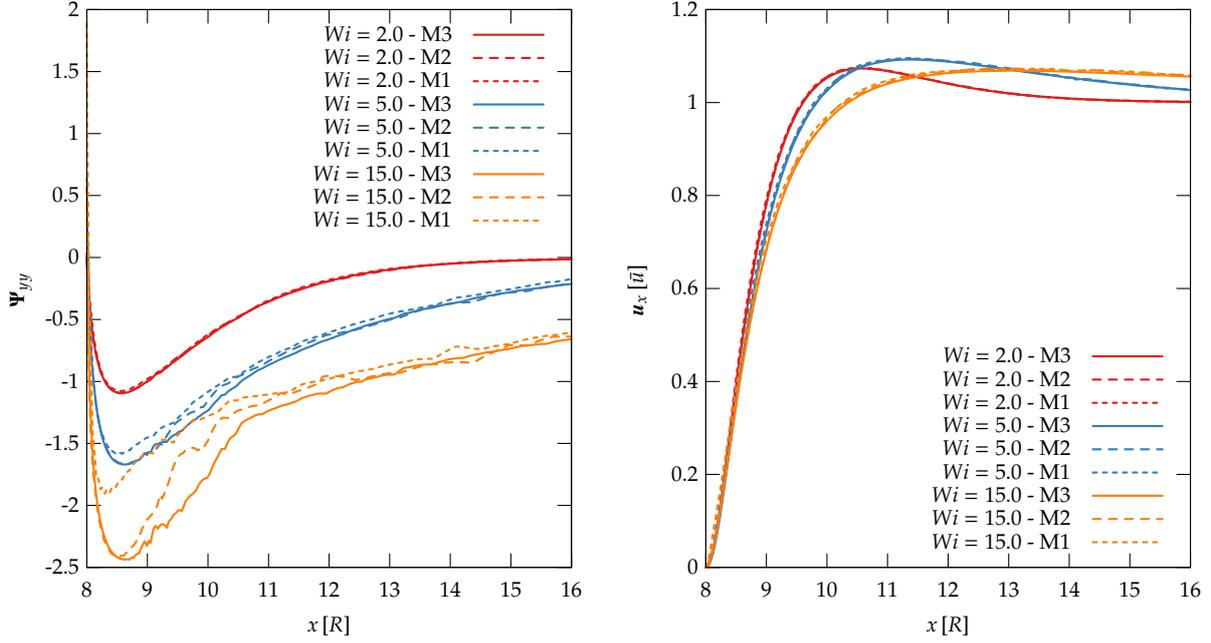

\footnotesize
\centering
\noindent
\begin{subfigure}[b]{0.49\textwidth}
\import{Fig/}{SpherePsiyyWake_a01.tex}
\end{subfigure}
\begin{subfigure}[b]{0.49\textwidth}
\import{Fig/}{SphereVelWake.tex}
\end{subfigure}
\caption{$\bm{\Psi}_{yy}$ and $\bm{u}_x$ plotted along the centerline in the wake of the sphere for $\alpha=0.1$.}
\label{fig:sphere_psiyy_vel}
\end{figure}

This effect becomes even more noticeable when considering the other degrees of freedom in our simulation. In Fig.~\ref{fig:sphere_psiyy_vel},
$\bm{\Psi}_{yy}$ has been plotted along the centerline for different Weissenberg numbers. The first point to notice is that
mesh convergence can be reached within the boundary-layer-adjusted mesh around the sphere, but as soon as the mesh resolution decreases,
the accuracy in the to-be-predicted degree of freedom $\bm{\Psi}_{yy}$ is lost. The impact becomes more severe the higher
the Weissenberg number is. In addition, by inspecting Fig.~\ref{fig:sphere_psiyy_vel}, it seems that for $Wi=15.0$, $\bm{\Psi}_{yy}$ exhibits
a small kink around $x=10.5\, R$ on Mesh M3, which may be attributable to a still insufficient refinement level of the mesh in that particular region.

On the other hand, the fact that $\bm{\Psi}_{yy}$ is negative also means that errors therein are exponentially damped in
$\bm{\sigma}_{yy}= \exp \bm{\Psi}_{yy}$. Since the latter is what essentially contributes to the momentum equation, it is not much
of a surprise that the velocity component depicted in Fig.~\ref{fig:sphere_psiyy_vel} is still smooth for all used meshes.
Furthermore, velocity overshoots exceeding $\bar{u}$ are clearly visible, in contrast to Oldroyd-B
simulations. Moreover, the downstream relaxation is once more delayed with increasing Weissenberg number.

\subsubsection{Performance of the Newton--Raphson algorithm}

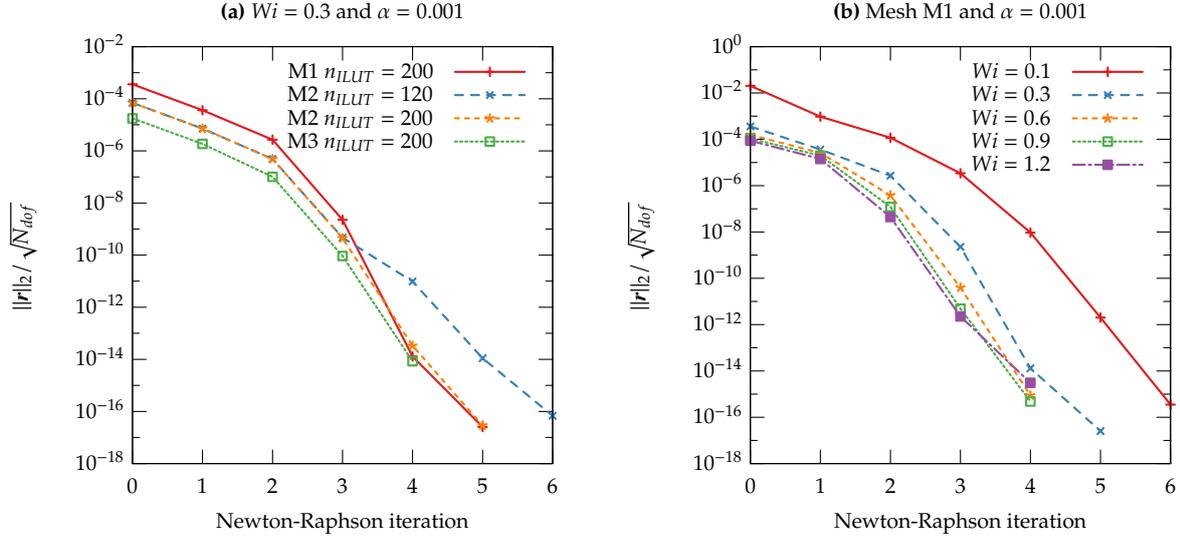
\begin{figure}[!t]
\footnotesize
\centering
\noindent
\begin{subfigure}[b]{0.49\textwidth}
\begin{tikzpicture}[gnuplot]
\gpcolor{color=gp lt color border}
\gpsetlinetype{gp lt border}
\gpsetlinewidth{1.00}
\draw[gp path] (3.116,0.985)--(3.296,0.985);
\draw[gp path] (8.644,0.985)--(8.464,0.985);
\node[gp node right] at (3.116,0.985) {$10^{-18}$};
\draw[gp path] (3.116,1.331)--(3.206,1.331);
\draw[gp path] (8.644,1.331)--(8.554,1.331);
\draw[gp path] (3.116,1.676)--(3.296,1.676);
\draw[gp path] (8.644,1.676)--(8.464,1.676);
\node[gp node right] at (3.116,1.676) {$10^{-16}$};
\draw[gp path] (3.116,2.022)--(3.206,2.022);
\draw[gp path] (8.644,2.022)--(8.554,2.022);
\draw[gp path] (3.116,2.367)--(3.296,2.367);
\draw[gp path] (8.644,2.367)--(8.464,2.367);
\node[gp node right] at (3.116,2.367) {$10^{-14}$};
\draw[gp path] (3.116,2.713)--(3.206,2.713);
\draw[gp path] (8.644,2.713)--(8.554,2.713);
\draw[gp path] (3.116,3.058)--(3.296,3.058);
\draw[gp path] (8.644,3.058)--(8.464,3.058);
\node[gp node right] at (3.116,3.058) {$10^{-12}$};
\draw[gp path] (3.116,3.404)--(3.206,3.404);
\draw[gp path] (8.644,3.404)--(8.554,3.404);
\draw[gp path] (3.116,3.749)--(3.296,3.749);
\draw[gp path] (8.644,3.749)--(8.464,3.749);
\node[gp node right] at (3.116,3.749) {$10^{-10}$};
\draw[gp path] (3.116,4.095)--(3.206,4.095);
\draw[gp path] (8.644,4.095)--(8.554,4.095);
\draw[gp path] (3.116,4.440)--(3.296,4.440);
\draw[gp path] (8.644,4.440)--(8.464,4.440);
\node[gp node right] at (3.116,4.440) {$10^{-8}$};
\draw[gp path] (3.116,4.786)--(3.206,4.786);
\draw[gp path] (8.644,4.786)--(8.554,4.786);
\draw[gp path] (3.116,5.131)--(3.296,5.131);
\draw[gp path] (8.644,5.131)--(8.464,5.131);
\node[gp node right] at (3.116,5.131) {$10^{-6}$};
\draw[gp path] (3.116,5.477)--(3.206,5.477);
\draw[gp path] (8.644,5.477)--(8.554,5.477);
\draw[gp path] (3.116,5.822)--(3.296,5.822);
\draw[gp path] (8.644,5.822)--(8.464,5.822);
\node[gp node right] at (3.116,5.822) {$10^{-4}$};
\draw[gp path] (3.116,6.168)--(3.206,6.168);
\draw[gp path] (8.644,6.168)--(8.554,6.168);
\draw[gp path] (3.116,6.513)--(3.296,6.513);
\draw[gp path] (8.644,6.513)--(8.464,6.513);
\node[gp node right] at (3.116,6.513) {$10^{-2}$};
\draw[gp path] (3.116,0.985)--(3.116,1.165);
\draw[gp path] (3.116,6.513)--(3.116,6.333);
\node[gp node center] at (3.116,0.677) { 0};
\draw[gp path] (4.037,0.985)--(4.037,1.165);
\draw[gp path] (4.037,6.513)--(4.037,6.333);
\node[gp node center] at (4.037,0.677) { 1};
\draw[gp path] (4.959,0.985)--(4.959,1.165);
\draw[gp path] (4.959,6.513)--(4.959,6.333);
\node[gp node center] at (4.959,0.677) { 2};
\draw[gp path] (5.880,0.985)--(5.880,1.165);
\draw[gp path] (5.880,6.513)--(5.880,6.333);
\node[gp node center] at (5.880,0.677) { 3};
\draw[gp path] (6.801,0.985)--(6.801,1.165);
\draw[gp path] (6.801,6.513)--(6.801,6.333);
\node[gp node center] at (6.801,0.677) { 4};
\draw[gp path] (7.723,0.985)--(7.723,1.165);
\draw[gp path] (7.723,6.513)--(7.723,6.333);
\node[gp node center] at (7.723,0.677) { 5};
\draw[gp path] (8.644,0.985)--(8.644,1.165);
\draw[gp path] (8.644,6.513)--(8.644,6.333);
\node[gp node center] at (8.644,0.677) { 6};
\draw[gp path] (3.116,6.513)--(3.116,0.985)--(8.644,0.985)--(8.644,6.513)--cycle;
\node[gp node center,rotate=-270] at (1.674,3.749) {$||\bm{r}||_2/\sqrt{N_{dof}}$};
\node[gp node center] at (5.880,0.215) {Newton-Raphson iteration};
\node[gp node center] at (5.880,6.975) {\textbf{(a)} $Wi=0.3$ and $\alpha=0.001$};
\node[gp node right] at (7.176,6.179) {M1 $n_{ILUT} = 200$};
\gpcolor{rgb color={0.894,0.102,0.110}}
\gpsetlinetype{gp lt plot 0}
\gpsetlinewidth{2.00}
\draw[gp path] (7.360,6.179)--(8.276,6.179);
\draw[gp path] (3.116,6.014)--(4.037,5.670)--(4.959,5.279)--(5.880,4.217)--(6.801,2.406)%
  --(7.723,1.470);
\gpsetpointsize{4.00}
\gppoint{gp mark 1}{(3.116,6.014)}
\gppoint{gp mark 1}{(4.037,5.670)}
\gppoint{gp mark 1}{(4.959,5.279)}
\gppoint{gp mark 1}{(5.880,4.217)}
\gppoint{gp mark 1}{(6.801,2.406)}
\gppoint{gp mark 1}{(7.723,1.470)}
\gppoint{gp mark 1}{(7.818,6.179)}
\gpcolor{color=gp lt color border}
\node[gp node right] at (7.176,5.871) {M2 $n_{ILUT} = 120$};
\gpcolor{rgb color={0.216,0.494,0.722}}
\gpsetlinetype{gp lt plot 1}
\draw[gp path] (7.360,5.871)--(8.276,5.871);
\draw[gp path] (3.116,5.768)--(4.037,5.427)--(4.959,5.024)--(5.880,3.981)--(6.801,3.400)%
  --(7.723,2.382)--(8.644,1.619);
\gppoint{gp mark 2}{(3.116,5.768)}
\gppoint{gp mark 2}{(4.037,5.427)}
\gppoint{gp mark 2}{(4.959,5.024)}
\gppoint{gp mark 2}{(5.880,3.981)}
\gppoint{gp mark 2}{(6.801,3.400)}
\gppoint{gp mark 2}{(7.723,2.382)}
\gppoint{gp mark 2}{(8.644,1.619)}
\gppoint{gp mark 2}{(7.818,5.871)}
\gpcolor{color=gp lt color border}
\node[gp node right] at (7.176,5.563) {M2 $n_{ILUT} = 200$};
\gpcolor{rgb color={1.000,0.498,0.000}}
\gpsetlinetype{gp lt plot 2}
\draw[gp path] (7.360,5.563)--(8.276,5.563);
\draw[gp path] (3.116,5.768)--(4.037,5.427)--(4.959,5.024)--(5.880,3.973)--(6.801,2.547)%
  --(7.723,1.491);
\gppoint{gp mark 3}{(3.116,5.768)}
\gppoint{gp mark 3}{(4.037,5.427)}
\gppoint{gp mark 3}{(4.959,5.024)}
\gppoint{gp mark 3}{(5.880,3.973)}
\gppoint{gp mark 3}{(6.801,2.547)}
\gppoint{gp mark 3}{(7.723,1.491)}
\gppoint{gp mark 3}{(7.818,5.563)}
\gpcolor{color=gp lt color border}
\node[gp node right] at (7.176,5.255) {M3 $n_{ILUT} = 200$};
\gpcolor{rgb color={0.302,0.686,0.290}}
\gpsetlinetype{gp lt plot 3}
\draw[gp path] (7.360,5.255)--(8.276,5.255);
\draw[gp path] (3.116,5.562)--(4.037,5.224)--(4.959,4.787)--(5.880,3.737)--(6.801,2.342);
\gppoint{gp mark 4}{(3.116,5.562)}
\gppoint{gp mark 4}{(4.037,5.224)}
\gppoint{gp mark 4}{(4.959,4.787)}
\gppoint{gp mark 4}{(5.880,3.737)}
\gppoint{gp mark 4}{(6.801,2.342)}
\gppoint{gp mark 4}{(7.818,5.255)}
\gpcolor{color=gp lt color border}
\gpsetlinetype{gp lt border}
\gpsetlinewidth{1.00}
\draw[gp path] (3.116,6.513)--(3.116,0.985)--(8.644,0.985)--(8.644,6.513)--cycle;
\gpdefrectangularnode{gp plot 1}{\pgfpoint{3.116cm}{0.985cm}}{\pgfpoint{8.644cm}{6.513cm}}
\end{tikzpicture}
\phantomsubcaption\label{fig:newtmeshcomp}
\end{subfigure}
\begin{subfigure}[b]{0.49\textwidth}
\begin{tikzpicture}[gnuplot]
\gpcolor{color=gp lt color border}
\gpsetlinetype{gp lt border}
\gpsetlinewidth{1.00}
\draw[gp path] (3.116,0.985)--(3.296,0.985);
\draw[gp path] (8.644,0.985)--(8.464,0.985);
\node[gp node right] at (3.116,0.985) {$10^{-18}$};
\draw[gp path] (3.116,1.292)--(3.206,1.292);
\draw[gp path] (8.644,1.292)--(8.554,1.292);
\draw[gp path] (3.116,1.599)--(3.296,1.599);
\draw[gp path] (8.644,1.599)--(8.464,1.599);
\node[gp node right] at (3.116,1.599) {$10^{-16}$};
\draw[gp path] (3.116,1.906)--(3.206,1.906);
\draw[gp path] (8.644,1.906)--(8.554,1.906);
\draw[gp path] (3.116,2.213)--(3.296,2.213);
\draw[gp path] (8.644,2.213)--(8.464,2.213);
\node[gp node right] at (3.116,2.213) {$10^{-14}$};
\draw[gp path] (3.116,2.521)--(3.206,2.521);
\draw[gp path] (8.644,2.521)--(8.554,2.521);
\draw[gp path] (3.116,2.828)--(3.296,2.828);
\draw[gp path] (8.644,2.828)--(8.464,2.828);
\node[gp node right] at (3.116,2.828) {$10^{-12}$};
\draw[gp path] (3.116,3.135)--(3.206,3.135);
\draw[gp path] (8.644,3.135)--(8.554,3.135);
\draw[gp path] (3.116,3.442)--(3.296,3.442);
\draw[gp path] (8.644,3.442)--(8.464,3.442);
\node[gp node right] at (3.116,3.442) {$10^{-10}$};
\draw[gp path] (3.116,3.749)--(3.206,3.749);
\draw[gp path] (8.644,3.749)--(8.554,3.749);
\draw[gp path] (3.116,4.056)--(3.296,4.056);
\draw[gp path] (8.644,4.056)--(8.464,4.056);
\node[gp node right] at (3.116,4.056) {$10^{-8}$};
\draw[gp path] (3.116,4.363)--(3.206,4.363);
\draw[gp path] (8.644,4.363)--(8.554,4.363);
\draw[gp path] (3.116,4.670)--(3.296,4.670);
\draw[gp path] (8.644,4.670)--(8.464,4.670);
\node[gp node right] at (3.116,4.670) {$10^{-6}$};
\draw[gp path] (3.116,4.977)--(3.206,4.977);
\draw[gp path] (8.644,4.977)--(8.554,4.977);
\draw[gp path] (3.116,5.285)--(3.296,5.285);
\draw[gp path] (8.644,5.285)--(8.464,5.285);
\node[gp node right] at (3.116,5.285) {$10^{-4}$};
\draw[gp path] (3.116,5.592)--(3.206,5.592);
\draw[gp path] (8.644,5.592)--(8.554,5.592);
\draw[gp path] (3.116,5.899)--(3.296,5.899);
\draw[gp path] (8.644,5.899)--(8.464,5.899);
\node[gp node right] at (3.116,5.899) {$10^{-2}$};
\draw[gp path] (3.116,6.206)--(3.206,6.206);
\draw[gp path] (8.644,6.206)--(8.554,6.206);
\draw[gp path] (3.116,6.513)--(3.296,6.513);
\draw[gp path] (8.644,6.513)--(8.464,6.513);
\node[gp node right] at (3.116,6.513) {$10^{0}$};
\draw[gp path] (3.116,0.985)--(3.116,1.165);
\draw[gp path] (3.116,6.513)--(3.116,6.333);
\node[gp node center] at (3.116,0.677) { 0};
\draw[gp path] (4.037,0.985)--(4.037,1.165);
\draw[gp path] (4.037,6.513)--(4.037,6.333);
\node[gp node center] at (4.037,0.677) { 1};
\draw[gp path] (4.959,0.985)--(4.959,1.165);
\draw[gp path] (4.959,6.513)--(4.959,6.333);
\node[gp node center] at (4.959,0.677) { 2};
\draw[gp path] (5.880,0.985)--(5.880,1.165);
\draw[gp path] (5.880,6.513)--(5.880,6.333);
\node[gp node center] at (5.880,0.677) { 3};
\draw[gp path] (6.801,0.985)--(6.801,1.165);
\draw[gp path] (6.801,6.513)--(6.801,6.333);
\node[gp node center] at (6.801,0.677) { 4};
\draw[gp path] (7.723,0.985)--(7.723,1.165);
\draw[gp path] (7.723,6.513)--(7.723,6.333);
\node[gp node center] at (7.723,0.677) { 5};
\draw[gp path] (8.644,0.985)--(8.644,1.165);
\draw[gp path] (8.644,6.513)--(8.644,6.333);
\node[gp node center] at (8.644,0.677) { 6};
\draw[gp path] (3.116,6.513)--(3.116,0.985)--(8.644,0.985)--(8.644,6.513)--cycle;
\node[gp node center,rotate=-270] at (1.674,3.749) {$||\bm{r}||_2/\sqrt{N_{dof}}$};
\node[gp node center] at (5.880,0.215) {Newton-Raphson iteration};
\node[gp node center] at (5.880,6.975) {\textbf{(b)} Mesh M1 and $\alpha=0.001$};
\node[gp node right] at (7.176,6.179) {$Wi=0.1$};
\gpcolor{rgb color={0.894,0.102,0.110}}
\gpsetlinetype{gp lt plot 0}
\gpsetlinewidth{2.00}
\draw[gp path] (7.360,6.179)--(8.276,6.179);
\draw[gp path] (3.116,5.994)--(4.037,5.585)--(4.959,5.305)--(5.880,4.834)--(6.801,4.048)%
  --(7.723,2.921)--(8.644,1.767);
\gpsetpointsize{4.00}
\gppoint{gp mark 1}{(3.116,5.994)}
\gppoint{gp mark 1}{(4.037,5.585)}
\gppoint{gp mark 1}{(4.959,5.305)}
\gppoint{gp mark 1}{(5.880,4.834)}
\gppoint{gp mark 1}{(6.801,4.048)}
\gppoint{gp mark 1}{(7.723,2.921)}
\gppoint{gp mark 1}{(8.644,1.767)}
\gppoint{gp mark 1}{(7.818,6.179)}
\gpcolor{color=gp lt color border}
\node[gp node right] at (7.176,5.871) {$Wi=0.3$};
\gpcolor{rgb color={0.216,0.494,0.722}}
\gpsetlinetype{gp lt plot 1}
\draw[gp path] (7.360,5.871)--(8.276,5.871);
\draw[gp path] (3.116,5.455)--(4.037,5.149)--(4.959,4.802)--(5.880,3.858)--(6.801,2.248)%
  --(7.723,1.416);
\gppoint{gp mark 2}{(3.116,5.455)}
\gppoint{gp mark 2}{(4.037,5.149)}
\gppoint{gp mark 2}{(4.959,4.802)}
\gppoint{gp mark 2}{(5.880,3.858)}
\gppoint{gp mark 2}{(6.801,2.248)}
\gppoint{gp mark 2}{(7.723,1.416)}
\gppoint{gp mark 2}{(7.818,5.871)}
\gpcolor{color=gp lt color border}
\node[gp node right] at (7.176,5.563) {$Wi=0.6$};
\gpcolor{rgb color={1.000,0.498,0.000}}
\gpsetlinetype{gp lt plot 2}
\draw[gp path] (7.360,5.563)--(8.276,5.563);
\draw[gp path] (3.116,5.344)--(4.037,5.104)--(4.959,4.542)--(5.880,3.318)--(6.801,1.889);
\gppoint{gp mark 3}{(3.116,5.344)}
\gppoint{gp mark 3}{(4.037,5.104)}
\gppoint{gp mark 3}{(4.959,4.542)}
\gppoint{gp mark 3}{(5.880,3.318)}
\gppoint{gp mark 3}{(6.801,1.889)}
\gppoint{gp mark 3}{(7.818,5.563)}
\gpcolor{color=gp lt color border}
\node[gp node right] at (7.176,5.255) {$Wi=0.9$};
\gpcolor{rgb color={0.302,0.686,0.290}}
\gpsetlinetype{gp lt plot 3}
\draw[gp path] (7.360,5.255)--(8.276,5.255);
\draw[gp path] (3.116,5.294)--(4.037,5.068)--(4.959,4.389)--(5.880,3.039)--(6.801,1.808);
\gppoint{gp mark 4}{(3.116,5.294)}
\gppoint{gp mark 4}{(4.037,5.068)}
\gppoint{gp mark 4}{(4.959,4.389)}
\gppoint{gp mark 4}{(5.880,3.039)}
\gppoint{gp mark 4}{(6.801,1.808)}
\gppoint{gp mark 4}{(7.818,5.255)}
\gpcolor{color=gp lt color border}
\node[gp node right] at (7.176,4.947) {$Wi=1.2$};
\gpcolor{rgb color={0.596,0.306,0.639}}
\gpsetlinetype{gp lt plot 4}
\draw[gp path] (7.360,4.947)--(8.276,4.947);
\draw[gp path] (3.116,5.266)--(4.037,5.024)--(4.959,4.254)--(5.880,2.936)--(6.801,2.056);
\gppoint{gp mark 5}{(3.116,5.266)}
\gppoint{gp mark 5}{(4.037,5.024)}
\gppoint{gp mark 5}{(4.959,4.254)}
\gppoint{gp mark 5}{(5.880,2.936)}
\gppoint{gp mark 5}{(6.801,2.056)}
\gppoint{gp mark 5}{(7.818,4.947)}
\gpcolor{color=gp lt color border}
\gpsetlinetype{gp lt border}
\gpsetlinewidth{1.00}
\draw[gp path] (3.116,6.513)--(3.116,0.985)--(8.644,0.985)--(8.644,6.513)--cycle;
\gpdefrectangularnode{gp plot 1}{\pgfpoint{3.116cm}{0.985cm}}{\pgfpoint{8.644cm}{6.513cm}}
\end{tikzpicture}
\phantomsubcaption\label{fig:newtM1}
\end{subfigure}
\caption{Convergence behavior of the Newton-Raphson algorithm for different settings.}
\label{fig:newtonraphson}
\end{figure}

It is important to note that the Newton--Raphson method indeed delivers quadratic convergence up to the point
where the errors emerging from the inexact linear solution or the limited floating-point accuracy become dominant.
Fig.~\ref{fig:newtonraphson} depicts a comparative study of the residual after each Newton-Raphson iteration across
different meshes, as well as an analysis of the convergence behavior for different Weissenberg numbers. In both cases,
the residual has been evaluated in the Euclidean norm and scaled by the square root of the total number of degrees of freedom
in order to make the results comparable across different mesh sizes. As one sees in Fig.~\ref{fig:newtmeshcomp}, the convergence
history can be roughly split into three phases: In the first step, the improvement is rather moderate; at most linear convergence
was obtained. In the second phase, e.g., for the Mesh M3, one observes a relative improvement of the residual
by a factor of $18.34$ in the second step
and an improvement by a factor of $1097 > (18.34)^2$ in the third step; thus slightly exceeding quadratic convergence.
It becomes apparent from Fig.~\ref{fig:newtmeshcomp} that
this convergence is mesh-independent. The third phase is then dominated by errors introduced by the inexact solution of the
linear equation systems. This can be deduced from the results in Fig.~\ref{fig:newtmeshcomp}, where the calculations on Mesh M2
were performed using two different ILUT fill-in settings. A reduction of the ILUT
fill-in, and a therewith increased error in the solution of the linear systems, directly leads to a deterioration of the
quadratic convergence to at most asymptotically linear convergence
\cite{Kelley1995}. Furthermore, in the last steps the convergence is limited by the fact
that a residual far beneath $10^{-16}$ is in general not attainable due to the
floating-point arithmetic used.

In Fig.~\ref{fig:newtM1},
one notices that, using the result obtained for the previously calculated Weissenberg number as an initial guess for the
subsequent calculation, the convergence progression is similar across the consecutive runs. Only the starting point $Wi=0.1$
does not fully fit into this picture, which on the one hand has to be attributed to the circumstance that for this case the
initial guess was set to zero in the interior of the computational domain,
and on the other hand is a consequence of the derivative of $\tau_{cons} (\bm{u}^h \cdot\nabla) \bm{\Phi}^h$
in the discretized weak form being neglected (cf. Section~\ref{sec:linearization}). The latter is a remedy for the fact that without these
additional terms, the iterative scheme seems to be more robust with regard to the choice of the initial guess.

\subsection{Sedimenting ellipsoid}

\begin{figure}[t]
\begin{center}
\includegraphics[width=\textwidth]{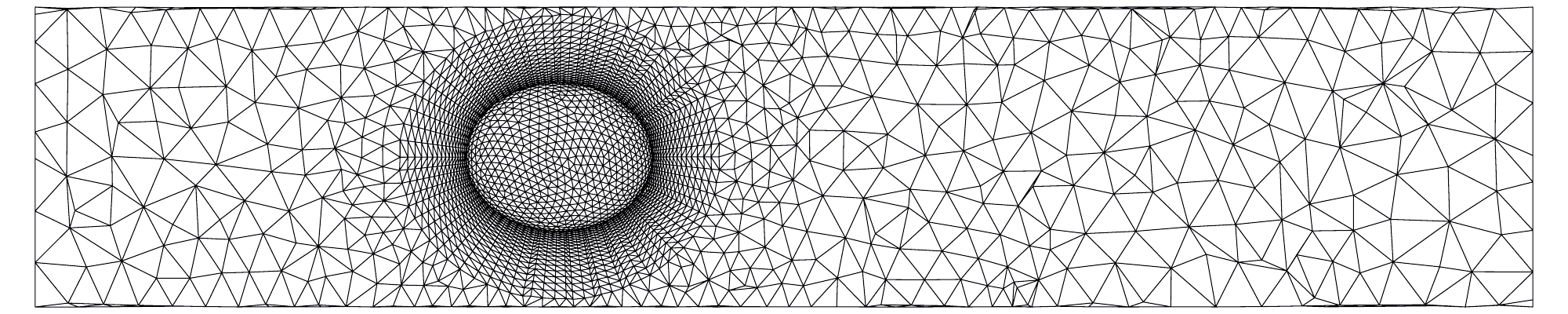}
\end{center}
\caption{Cut through the $xy$-plane of Mesh M4, as used in the calculations with an ellipsoid.}
\label{fig:cutmeshm4}
\end{figure}

\begin{table}[t]
\footnotesize
\centering
\begin{tabular}{lc}
\toprule
 & M4 \\
\midrule
Number of elements on the ellipsoid & 2706 \\
Total number of nodes & 218663 \\
Total number of elements & 158353 \\
Krylov-space dimension & 300 \\
ILUT maximal fill-in $n_{ILUT}$ & 200 \\
ILUT threshold & $10^{-4}$ \\
Number of cores & 256\\
\bottomrule
\end{tabular}
\caption{Mesh and solver attributes used for the sedimenting ellipsoid benchmark.}
\label{tab:meshproperties_ellipsoid}
\end{table}

In order to demonstrate the applicability of the proposed method to a truly three-dimensional problem,
a case similar to the sedimenting sphere benchmark is considered, but with the sphere replaced by a tri-axial ellipsoid. The latter was
chosen with the semi-principal axes aligned to the coordinate axes. The length of the axes in $x$,$y$, and $z$ direction
were set to $a=1.25\, R$, $b=1.0\, R$, and $c=0.8\, R$, respectively, such that in accordance with the findings in \cite{Huang1998},
the semi-major axis coincides with the main flow direction for $Re=0$. The tube radius was kept as $2R$.

The mesh, as depicted in Fig.~\ref{fig:cutmeshm4}, was chosen similar to the Mesh M2 in the sedimenting sphere benchmark,
which already provided a good
trade-off between computational cost and accuracy in the drag computation. Therefore, the GMRES/ILUT parameters were also
chosen accordingly, as can be seen in Tab.\ref{tab:meshproperties_ellipsoid}.

\begin{table}[t]
\footnotesize
\centering
\begin{tabular}{ccccc}
\toprule
\multirow{2}[3]{*}{$Wi$} & \multicolumn{4}{c}{$K$}\\
\cmidrule(lr){2-5}
 & $\alpha=0$ & $\alpha=0.001$ & $\alpha=0.01$ & $\alpha=0.1$ \\
\midrule
0.1 & 4.90847 & 4.90782 & 4.90211 & 4.85331\\
0.2 & 4.85959 & 4.85715 & 4.83621 & 4.68884\\
0.3 & 4.79819 & 4.79305 & 4.75092 & 4.50990\\
0.4 & 4.73648 & 4.72791 & 4.66130 & 4.34532\\
0.5 & 4.68082 & 4.66810 & 4.57513 & 4.20056\\
0.6 & 4.63365 & 4.61608 & 4.49540 & 4.07438\\
0.7 & 4.59537 & 4.57217 & 4.42252 & 3.96425\\
0.8 & 4.56535 & 4.53566 & 4.35581 & 3.86772\\
0.9 & 4.54252 & 4.50538 & 4.29432 & 3.78269\\
1.0 & 4.52568 & 4.48007 & 4.23718 & 3.70739\\
1.1 & 4.51370 & 4.45854 & 4.18372 & 3.64038\\
1.2 & 4.50554 & 4.43976 & 4.13345 & 3.58045\\
1.3 & 4.50034 & 4.42286 & 4.08604 & 3.52658\\
1.4 & 4.49738 & 4.40717 & 4.04124 & 3.47795\\
1.5 &               & 4.39215 & 3.99883 & 3.43385\\
1.6 &               & 4.37747 & 3.95867 & 3.39369\\
1.8 &               &               & 3.88450 & 3.32332\\
2.0 &               &               & 3.81772 & 3.26370\\
$\vdots$ & & & $\vdots$ & $\vdots$ \\
10.0&              &               & 3.01520 & 2.71617\\
11.0&              &               & 2.98506 & 2.69985\\
12.0&              &               & 2.95902 & 2.68597\\
13.0&              &               & 2.93626 & 2.67400\\
14.0&              &               & 2.91617 & 2.66356\\
15.0&              &               & 2.89827 & 2.65437\\
\bottomrule
\end{tabular}
\caption{Results for the correction factor $K$ of the drag on the ellipsoid.}
\label{tab:dragellipsoid}
\end{table}

Our main objective of the investigation was the drag correction factor $K$, where the latter has been defined for the sake of simplicity
as in the case of the falling sphere, cf. Eq.~\eqref{eqn:dragcorrection}. Nonetheless,
the Stokesian drag formula can be generalized to ellipsoids in principle \cite{Lamb1932}. The results in Tab.~\ref{tab:dragellipsoid}
confirm the general trend of the simulations with the spherical geometry:
The drag decreases monotonically with increasing Weissenberg number.
It can also be stated that the general drag level is below the drag levels obtained in the simulations
with a sphere as obstacle, which may be attributed to the reduced cross section. With increasing $\alpha$,
higher Weissenberg numbers can be attained, and the effect of reduced drag due to increased shear-thinning becomes visible.

\section{Conclusion and discussion}
The main objective of this paper was to derive a log-conformation formulation that on the one hand inherits the stability properties
of the originally proposed log-conformation formulation \cite{Fattal2004}, but on the other hand also paves the way for an application of Newton's method
in numerical simulations. Furthermore, we especially sought a description that could be applied in three
dimensions with the same ease as the previously published two-dimensional approaches \cite{Knechtges2014,Saramito2014}.

To demonstrate the numerical benefit of this approach,
we implemented a proof-of-concept three-dimensional finite element solver and subsequently
tested it by means of the sedimenting sphere and ellipsoid benchmarks.
The simulations exhibited the best-possible convergence properties of quadratic-convergence.

Since the new constitutive equations are just a rewording of the original
log-conformation equations, the proposed formulation cannot further improve the stability.
As such, we were not able to obtain results
beyond a Weissenberg number of $Wi=1.4$ for a sphere sedimenting through an Oldroyd-B fluid. Since switching to the Giesekus model removed
this limitation, the characteristic behavior of the Oldroyd-B fluid in extensional flow regimes might be the underlying
reason for this restriction.

In addition to the just-mentioned advantages for the numerical application, our formulation
is intrinsically defined in an undiscretized setting, which may reveal new perspectives on the analytical properties
of the used constitutive models in the future.
In particular, the seamless incorporation of the so-called free-energy estimates in the log-conformation formulation \cite{Boyaval2009}
and their application to the global-in-time existence of solutions \cite{Masmoudi2011,Barrett2011} may give new insights.

\section{Acknowledgments}
The author gratefully acknowledges support from the German Research Foundation (DFG) grant
”Computation of Die Swell Behind a Complex Profile Extrusion Die Using a Stabilized Finite Element Method for Various Thermoplastic Polymers”
and the DFG program GSC 111 (AICES Graduate School).
The computations were conducted on computing clusters provided by the Jülich Aachen Research Alliance (JARA). Furthermore, I want to thank
Marek Behr, Stefanie Elgeti, and Stefan Haßler for their indispensable remarks during the preparation of the manuscript.

\appendix

\bibliographystyle{elsarticle-num}
\bibliography{references}


\end{document}